\documentclass[a4paper,11pt,draft]{amsart}

\usepackage[T1]{fontenc}
\usepackage[utf8]{inputenc}
\usepackage{txfonts}
\usepackage{mathrsfs}

\usepackage{amsmath,amssymb,amsthm}
	\theoremstyle{plain}
	\newtheorem{thm}{Theorem}[section]
	\newtheorem{lem}[thm]{Lemma}
	\newtheorem{cor}[thm]{Corollary}
	\newtheorem{prop}[thm]{Proposition}
	
	\theoremstyle{definition}
	\newtheorem{dfn}[thm]{Definition}

	\theoremstyle{remark}
	\newtheorem{rem}[thm]{Remark}
	\newtheorem{expl}[thm]{Example}

\usepackage{enumitem}
	\setlist[enumerate,1]{leftmargin=0pt,itemindent=17pt,label=\textup{(\arabic*)}}
	\setlist[itemize]{leftmargin=*}

\usepackage{xypic}

\usepackage{mathtools}

\usepackage{comment}

\begin{document}
\title{On a generalized Fraïssé limit construction}
\author[S.~Masumoto]{Shuehi MASUMOTO}
\address[S.~Masumoto]{Graduate~School~of~Mathematical~Sciences, the~University~of~Tokyo}
\email{masumoto@ms.u-tokyo.ac.jp}
\keywords{Fraïssé theory; Metric structures}

\begin{abstract}
In this paper, we present a slightly modified version of 
Fraïssé theory which is used in~\cite{eagle16:_fraisse_limits} 
and~\cite{masumoto16:_jiang_su}.  
Using this version, we also show that 
every UHF algebra can be recognized as 
a Fraïssé limit of a class of 
C*-algebras of matrix-valued continuous functions on cubes with distinguished traces.  
\end{abstract}

\maketitle


\section{Introduction}

Fraïssé theory was originally invented by Rolland Fraïssé 
in~\cite{fraisse54:_extension_relations}.  
The fundamental theorem of this theory claims that 
there is a bijective correspondence between 
the ultra-homogeneous structures and what we call Fraïssé classes.  
For an ultra-homogeneous structure, the corresponding Fraïssé class is 
its age, that is, the class of all finitely generated substructures; 
and the ultra-homogeneous structure is recovered as 
a generic inductive limit of members of the Fraïssé class, 
so that it is called the Fraïssé limit of the class.  

This theory has been generalized to the setting of 
metric structures (\cite{schoretsanitis07:_fraisse_theory}, 
\cite{yaacov15:_fraisse_limits}).  
In~\cite{yaacov15:_fraisse_limits}, 
a Fraïssé class of metric structures was defined as 
a class of finitely generated metric structures which satisfies 
the axioms called the hereditary property (HP), 
the joint embedding property (JEP), 
the near amalgamation property (NAP), 
the Polish property (PP), 
and the continuity property (CP).  
Then it was shown that 
there is a bijective correspondence between 
the approximately ultra-homogeneous structures 
and the Fraïssé classes as above.  

The key idea of the proof of the fundamental theorem in~\cite{yaacov15:_fraisse_limits}  
was to use approximate isomorphisms.  
If $\mathscr{K}$ is a Fraïssé class, 
then a structure $\mathcal{M}$ is called a $\mathscr{K}$-structure 
if its age is included in $\mathscr{K}$; 
and an approximate $\mathscr{K}$-isomorphism from a $\mathscr{K}$-structure 
$\mathcal{M}$ into another $\mathscr{K}$-structure $\mathcal{N}$ 
is defined as a bi-Kat\v{e}tov map from 
$|\mathcal{M}| \times |\mathcal{N}|$ to $[0, \infty]$ 
which approximately dominates a map of the form 
$(a, b) \mapsto d(\iota(a), \eta(b))$, 
where $\iota$ and $\eta$ are finite partial embeddings of 
$\mathcal{M}$ and $\mathcal{N}$ into some $\mathcal{C} \in \mathscr{K}$, respectively.  
Intuitively, an approximate $\mathscr{K}$-isomorphism 
from $\mathcal{M}$ into $\mathcal{N}$ should be 
thought of as a condition to be imposed on 
an embedding of a substructure of $\mathcal{M}$ into an extension of 
a substructure of $\mathcal{N}$.  
What is important is that we can consider compositions of 
approximate isomorphisms.  
Namely, if $\varphi$ is an approximate isomorphism from $\mathcal{M}_1$ 
into $\mathcal{M}_2$ and $\psi$ is an approximate isomorphism from 
$\mathcal{M}_2$ into $\mathcal{M}_3$, 
then one can consider a new approximate isomorphism 
$\psi\varphi$ from $\mathcal{M}_1$ into $\mathcal{M}_3$.  
Thanks to this fact, 
we can prove that the limits of two generic inductive systems of
members of a Fraïssé class $\mathscr{K}$ are isomorphic to each other and are 
ultra-homogeneous, 
by carrying over a back-and-forth argument between them via approximate isomorphisms.  

In~\cite{eagle16:_fraisse_limits}, 
a more relaxed version of Fraïssé theory was considered 
in order to recognize several well-known examples of operator algebras 
as Fraïssé limits.  
In their definition of Fraïssé classes, 
the axioms PP and CP were replaced by weaker conditions 
called the weak Polish property (WPP) and the Cauchy continuity property (CCP), 
and the axiom HP was omitted.  
Corresponding to this change, 
the definition of $\mathscr{K}$-structures for a Fraïssé class $\mathscr{K}$ 
was also modified: 
a structure $\mathcal{M}$ is said to be a $\mathscr{K}$-structure 
if it is an inductive limit of members of $\mathscr{K}$.    
Then it was claimed that every Fraïssé class has its limit, 
that is, for every Fraïssé class $\mathscr{K}$, 
there exists a unique $\mathscr{K}$-structure $\mathcal{M}$ which is 
approximately $\mathscr{K}$-ultra-homogeneous 
in the sense that if $\iota_1, \iota_2$ are two embeddings of 
a member $\mathcal{A}$ of $\mathscr{K}$, 
then for any finite subset $F \subseteq |\mathcal{A}|$ and any $\varepsilon > 0$ 
there exists an automorphism $\alpha$ of $\mathcal{M}$ 
with $d(\alpha \circ \iota_1(a), \iota_2(a)) < \varepsilon$ for all $a \in F$.  
The proof of this claim was not presented in~\cite{eagle16:_fraisse_limits}, 
because it was thought that 
the proof of~\cite{yaacov15:_fraisse_limits} would still work almost verbatim 
in this setting.  

In order to adopt the proof of~\cite{yaacov15:_fraisse_limits}, 
we first have to guess what is the appropriate definition of 
approximate isomorphisms in this relaxed setting.  
One of the candidates would be the same definition as the original one 
in~\cite{yaacov15:_fraisse_limits}, 
but this does not seem to work, because within this setting 
one can no longer prove in the same way as~\cite{yaacov15:_fraisse_limits} 
that compositions of approximate isomorphisms between $\mathscr{K}$-structures 
are approximate isomorphisms (see Remark~\ref{rem:_the_gap}).  
This is a fatal gap, 
because this property is used to establish the existence and uniqueness of 
a limit of a Fraïssé class.  

In this paper, we reconstruct the theory presented 
in~\cite{eagle16:_fraisse_limits} and reveal the correct form of 
the fundamental theorem.  
Because of the gap explained in the previous paragraph, 
it turns out that the concept of $\mathscr{K}$-structures should have been 
more complicated, and the homogeneity property the generic limit satisfies 
is weaker than the original claim.  

Simultaneously, we slightly generalize the theory so that we can deal with categories.  
The motivation of this generalization is the following.  
In~\cite{eagle16:_fraisse_limits}, the Jiang--Su algebra $\mathcal{Z}$ was 
recognized as a Fraïssé limit of the class of prime dimension drop algebras 
with distinguished faithful traces.  An alternative proof of this fact 
was given in~\cite{masumoto16:_jiang_su}, which was based on 
the fact that every unital embedding between dimension drop algebras is 
approximately diagonalizable.  
These results led to the prospect of giving a short proof of the fact
that the Jiang--Su algebra is tensorially self-absorbing 
(i.e.~$\mathcal{Z} \otimes \mathcal{Z} \simeq \mathcal{Z}$), 
and the first step of such a short proof was expected to be showing that 
the class of tensor products of dimension drop algebras with distinguished faithful 
traces is also a Fraïssé class.  
For this, if we adopt the same strategy as~\cite{masumoto16:_jiang_su}, 
then we should first show the statement that every unital embedding between 
tensor products of dimension drop algebras is approximately diagonalizable, 
which turns out to be false (cf.~Remark~\ref{rem:_counterexample}).  
A natural solution to this problem would be simply restricting embeddings to 
approximately diagonalizable ones, so that the object under consideration 
would not be a class but a category.  

We should note that there is another research on 
Fraïssé theory for categories by Wies\l aw Kubi\'{s}~\cite{kubis13:_metric-enriched}.  
His theory is developed within the theory of categories enriched over metric spaces, 
while our approach is based on the work by Itaï Ben Yaacov~\cite{yaacov15:_fraisse_limits}.  

This paper is organized as follows.  
In the second section, we recall the definition and properties of approximate isometries.  
The third section is devoted to metric structures and approximate isomorphisms.  
The existence and uniqueness of the Fraïssé limit is shown in the fourth section.  
The fifth section contains an application of this theory to UHF algebras.


\section{Approximate isometries}
In this section, we recall the definition and properties of approximate isometries.  
Our handling of them is based on~\cite{yaacov15:_fraisse_limits}.  
Proofs are reproduced for the convenience of the reader.  

Let $X$ and $Y$ be metric spaces.  
We denote by $\operatorname{JE}(X, Y)$ the set of all pairs 
$(\iota, \eta)$, where $\iota \colon X \to Z$ and $\eta \colon Y \to Z$ 
are isometries into some metric space $Z$.  
Each element of $\operatorname{JE}(X, Y)$ is called a \emph{joint embedding} of $X$ and $Y$.  

\begin{dfn}\label{dfn:_approximate_isometries}
\begin{enumerate}
\item
	Let $X$ be a metric space.  
	A map $\varphi \colon X \to [0, \infty]$ is said to be \emph{Kat\v{e}tov} 
	if it satisfies the inequalities 
	\[
	\begin{aligned}
		\varphi(x)  &\leq d_X(x, x') + \varphi(x'), & d_X(x, x') \leq \varphi(x) + \varphi(x')
	\end{aligned}
	\]
	for all $x, x' \in X$.  
\item
	Suppose that $X$ and $Y$ are metric spaces.  
	An \emph{approximate isometry} from $X$ to $Y$ is 
	a map $\varphi \colon X \times Y \to [0, \infty]$ which is separately Kat\v{e}tov.  
	The class of all approximate isometries from $X$ to $Y$ is denoted by 
	$\operatorname{Apx}(X, Y)$.  
	Note that, being a closed subset of $[0, \infty]^{X \times Y}$, 
	the space $\operatorname{Apx}(X, Y)$ is compact and Hausdorff 
	with respect to the topology of pointwise convergence.  
\end{enumerate}
\end{dfn}

Intuitively, an approximate isometry is a condition to be imposed on joint embeddings.  
A joint embedding $(\iota, \eta) \in \operatorname{JE}(X, Y)$ is said to 
\emph{satisfy} an approximate isometry $\varphi$ from $X$ to $Y$ 
if the inequality 
\[
	d\bigl(\iota(x), \eta(y)\bigr) \leq \varphi(x, y)
\]
holds for all $x \in X$ and $y \in Y$.  
We shall denote by $\operatorname{JE}^{\leq\varphi}(X, Y)$ 
the class of all joint embeddings satisfying $\varphi$.  
Clearly, the condition $\varphi \equiv \infty$ is the weakest condition.  
Note that if an approximate isometry $\varphi$ from $X$ to $Y$ takes a finite value at 
some point, then it is real-valued, because if $\varphi(x_0, y_0) < \infty$, then
\[
	\varphi(x, y) \leq d_X(x, x_0) + \varphi(x_0, y_0) + d(y_0, y) < \infty.  
\]

\begin{expl}
\begin{enumerate}
\item
	For a joint embedding $(\iota, \eta)$ of $X$ and $Y$, 
	the map 
	\[
		(x, y) \mapsto d\bigl(\iota(x), \eta(y)\bigr)
	\] 
	itself is an approximate isometry.  
	We shall denote this approximate isometry by $\varphi_{\iota, \eta}$.  
	If $\eta$ is equal to $\operatorname{id}_Y$, then $\varphi_{\iota, \eta}$ is 
	simply written as $\varphi_\iota$.  
	
	We shall show that every approximate isometry is of this form 
	unless it is equal to $\infty$.  
	To see this, let $\varphi \colon X \times Y \to [0, \infty)$ be an approximate isometry 
	and define a symmetric function $\delta \colon (X \coprod Y)^2 \to [0, \infty)$ by 
	\[
		\delta(z, z') = \left\{
		\begin{array}{ll}
			d_X(z, z') & \text{if $z, z' \in X$}, \\
			\varphi(z, z') & \text{if $z \in X$ and $z' \in Y$}, \\
			d_Y(z, z') & \text{if $z, z' \in Y$}.  
		\end{array}
		\right.
	\]
	Then it is easy to see that $\delta$ is a pseudo-metric.  
	If $\iota$ and $\eta$ are canonical embeddings of $X$ and $Y$ into 
	the quotient metric space $X \coprod_\varphi Y$, 
	then $d\bigl(\iota(x), \eta(y)\bigr) = \varphi(x, y)$, as desired.  
	It follows that, for any approximate isometries $\varphi$ and $\psi$ from $X$ to $Y$, 
	the inequality $\varphi \leq \psi$ holds if and only if 
	$\operatorname{JE}^{\leq\varphi}(X, Y)$ is included in $\operatorname{JE}^{\leq\psi}(X, Y)$, 
	so the order $\leq$ completely reflects the strength of conditions.  

	Note that a net $\{\iota_\alpha\}$ of isometries from $X$ into $Y$ converges pointwise to 
	an isometry $\iota$ if and only if $\{\varphi_{\iota_\alpha}\}$ converges to $\varphi_\iota$.  
	Indeed, if $\{\iota_\alpha\}$ converges to $\iota$, then 
	\[
		\varphi_{\iota_\alpha}(x, y) = d\bigl(\iota_\alpha(x), y\bigr) 
		\to d\bigl(\iota(x), y\bigr) = \varphi_\iota(x, y)
	\]
	for all $x \in X$ and $y \in Y$.  
	Conversely, if $\{\varphi_{\iota_\alpha}\}$ converges to $\varphi_\iota$, 
	then for any $x \in X$ we have 
	\[
		d\bigl(\iota_\alpha(x), \iota(x)\bigr) = \varphi_{\iota_\alpha}\bigl(x, \iota(x)\bigr) 
		\to \varphi_\iota\bigl(x, \iota(x)\bigr) = d\bigl(\iota(x), \iota(x)\bigr) = 0.  
	\]
\item
	If $\varphi$ is an approximate isometry from $X$ to $Y$, then 
	\[
		\varphi^*(y, x) := \varphi(x, y)
	\]
	defines an approximate isometry $\varphi^*$ from $Y$ to $X$.  
\item
	Given $\varphi \in \operatorname{Apx}(X, Y)$ and $\psi \in \operatorname{Apx}(Y, Z)$, 
	we define their \emph{composition} by 
	\[
		\psi\varphi(x, z) := \inf_{y \in Y}\bigl(\varphi(x, y) + \psi(y, z)\bigr).  
	\]
	Here, we shall check that $\psi\varphi$ is an approximate isometry from $X$ to $Z$.  
	Indeed, if $x$ and $x'$ are points of $X$, then 
	\begin{align*}
		\psi\varphi(x, z) 
		&= \inf_{y \in Y}\bigl(\varphi(x, y) + \psi(y, z)\bigr) \\
		&\leq \inf_{y \in Y}\bigl(d(x, x') + \varphi(x', y) + \psi(y, z)\bigr) \\
		&= d(x, x') + \psi\varphi(x' z)
	\shortintertext{and}
		d(x, x') 
		&\leq \inf_{\mathclap{y, y' \in Y}}
		\kern 4pt \bigl(\varphi(x, y) + d(y, y') + \varphi(x', y')\bigr) \\
		&\leq \inf_{\mathclap{y, y' \in Y}}
		\kern 4pt \bigl(\varphi(x, y) + \psi(y, z) + \varphi(x', y') + \psi(y', z)\bigr) \\
		&= \psi\varphi(x, z) + \psi\varphi(x' z), 
	\end{align*}
	so $\psi\varphi({}\cdot{}, z)$ is Kat\v{e}tov for all $z \in Z$.  
	By symmetry $\psi\varphi(x, {}\cdot{})$ is also Kat\v{e}tov for all $x \in X$, 
	so $\psi\varphi$ is an approximate isometry.  
	
	It is worth noting that if $(\iota_1, \iota_2) \in \operatorname{JE}^{\leq\varphi}(X, Y)$ 
	and $(\iota_2, \iota_3) \in \operatorname{JE}^{\leq\psi}(Y, Z)$, 
	then $(\iota_1, \iota_3) \in \operatorname{JE}^{\leq\psi\varphi}(X, Z)$, 
	and $\psi\varphi$ is the smallest approximate isometry satisfying this property.  
	Also, it can be easily seen that the equality 
	$\varphi_{\iota, \eta}^{} = \varphi_\eta^*\varphi_\iota^{}$ holds 
	for any joint embedding $(\iota, \eta)$.  
\item
	Let $X' \subseteq X$ and $Y' \subseteq Y$ be subspaces.  
	If $\varphi$ is an approximate isometry from $X$ to $Y$, 
	then its restriction $\varphi|_{X' \times Y'}$ is 
	an approximate isometry from $X'$ to $Y'$.  
	Note that, if $\iota \colon X' \to X$ and $\eta \colon Y' \to Y$ are
	the canonical embeddings, then $\varphi|_{X' \times Y'}$ is equal to 
	$\varphi_\eta^*\varphi\varphi_\iota^{}$.  
	Now suppose that $\psi$ is an approximate isometry from $X'$ to $Y'$.  
	The \emph{trivial extension} of $\psi$ to $X \times Y$ is defined by 
	$\psi|^{X \times Y} := \varphi_\eta^{}\psi\varphi_\iota^*$.  
	It is easy to show that $\psi|^{X \times Y}$ is the largest 
	approximate isometry such that the restriction to $X' \times Y'$ is equal to $\psi$.  
	More generally, an approximate isometry $\theta$ from $X$ to $Y$ satisfies 
	$\theta \leq \psi|^{X \times Y}$ if and only if $\theta|_{X' \times Y'} \leq \psi$.  
\item
	If $\varphi$ is an approximate isometry from $X$ to $Y$ and $\varepsilon$ is 
	a non-negative real number, then the \emph{relaxation} of $\varphi$ by $\varepsilon$ 
	is defined by $(x, y) \mapsto \varphi(x, y) + \varepsilon$.  
	We simply denote this approximate isometry by $\varphi + \varepsilon$.  
	Note that the operation of taking relaxations commutes with compositions.  
\end{enumerate}
\end{expl}

\begin{dfn}\label{dfn:_epsilon-totality}
An approximate isometry $\varphi$ from $X$ to $Y$ is said to be 
\begin{itemize}
\item
	\emph{$\varepsilon$-total} if 
	$\varphi^*\varphi \leq \varphi_{\operatorname{id}_X} + 2\varepsilon$.  
\item
	\emph{$\varepsilon$-surjective} if $\varphi^*$ is $\varepsilon$-total.  
\item
	\emph{$\varepsilon$-bijective} if $\varphi$ is $\varepsilon$-total and 
	$\varepsilon$-surjective.  
\end{itemize}
\end{dfn}

If $\varphi$ and $\psi$ are approximate isometries from $X$ to $Y$ 
with $\psi \leq \varphi$, 
then clearly $\psi^*\psi \leq \varphi^*\varphi$.  
Therefore, if $\varphi$ is $\varepsilon$-total, then so is $\psi$.  
Similarly, if $\varphi$ is $\varepsilon$-surjective, then so is $\psi$.  

\begin{prop}\label{prop:_characterization_of_epsilon-totality}
An approximate isometry $\varphi$ from $X$ to $Y$ is $\varepsilon$-total 
if and only if any $(\iota, \eta) \in \operatorname{JE}^{\leq\varphi}(X, Y)$ 
satisfies $d\bigl(\iota(x), \eta[Y]\bigr) \leq \varepsilon$ for each $x \in X$.  
In particular, if $Y$ is complete and $\varphi$ is $\varepsilon$-total for any $\varepsilon$, 
then it is of the form $\varphi_\iota$ for a unique isometry $\iota \colon X \to Y$.  
\end{prop}

\begin{proof}
Suppose that $\varphi$ is $\varepsilon$-total and let $(\iota, \eta)$ be 
in $\operatorname{JE}^{\leq\varphi}(X, Y)$.  
Then, for any $x \in X$, we have 
\[
\begin{aligned}
	2\inf_{y \in Y} d\bigl(\iota(x), \eta(y)\bigr) 
	&\leq \inf_{y \in Y}\bigl(\varphi(x, y) + \varphi^*(y, x)\bigr) 
	= \varphi^*\varphi(x, x) \\
	&\leq \varphi_{\operatorname{id}}(x, x) + 2\varepsilon = 2\varepsilon, 
\end{aligned}
\]
so $d\bigl(\iota(x), \eta[Y]\bigr) \leq \varepsilon$.  

Conversely, suppose that $d\bigl(\iota(x), \eta[Y]) \leq \varepsilon$ holds for any 
$(\iota, \eta) \in \operatorname{JE}^{\leq\varphi}(X, Y)$ and any $x \in X$.  
Then $\varphi \not\equiv \infty$, so it is of the form 
$\varphi_{\iota, \eta}$, and 
\[
\begin{aligned}
	\varphi^*\varphi(x, x') 
	&= \inf_{y \in Y} \bigl(d\bigl(\iota(x), \eta(y)\bigr) 
	+ d\bigl(\eta(y), \iota(x')\bigr)\bigr) \\
	&\leq d(x, x') + 2\inf_{y \in Y}d\bigl(\iota(x), \eta(y)\bigr) \\
	&\leq \varphi_{\operatorname{id}}(x, x') + 2 \varepsilon.  
\end{aligned}
\]
\end{proof}

Let $\varphi$ be an approximate isometry from $X$ to $Y$.  
We set 
\[
	\operatorname{Apx}^{\leq\varphi}(X, Y) 
	:= \{\psi \in \operatorname{Apx}(X, Y) \mid \psi \leq \varphi\}.  
\]
We also denote by $\operatorname{Apx}^{\vartriangleleft\varphi}(X, Y)$ 
the interior of the closed subset $\operatorname{Apx}^{\leq\varphi}(X, Y)$ 
of the compact Hausdorff space $\operatorname{Apx}(X, Y)$, 
and write $\psi \vartriangleleft \varphi$ or $\varphi \vartriangleright \psi$ 
if $\psi$ belongs to $\operatorname{Apx}^{\vartriangleleft\varphi}(X, Y)$.  
If $\operatorname{Apx}^{\vartriangleleft\varphi}(X, Y)$ is nonempty, 
then $\varphi$ is said to be \emph{strict}.  
The class of all strict approximate isometries is denoted by $\operatorname{Stx}(X, Y)$.  

It can be easily verified that the relation $\vartriangleleft$ is preserved under 
restrictions and trivial extensions.  In particular, 
restrictions and trivial extensions of a strict approximate isometries are strict.  

\begin{prop}\label{prop:_characterization_of_strict_domination}
For $\varphi, \psi \in \operatorname{Apx}(X, Y)$, the following are equivalent.  
\begin{enumerate}[label=\textup{(\roman*)}]
\item
	The relation $\psi \vartriangleleft \varphi$ holds.  
\item
	There exist finite subsets $X_0 \subseteq X$ and $Y_0 \subseteq Y$ 
	and a positive real number $\varepsilon$ such that the inequality 
	$\varphi \geq (\psi|_{X_0 \times Y_0})|^{X \times Y} + \varepsilon$ holds.  
\item
	Same as \textup{(ii)}, with $\geq$ replaced by $\vartriangleright$.  
\end{enumerate}
Moreover, if these conditions are satisfied, then there exist 
then there exist finite subsets $X_0 \subseteq X$ and $Y_0 \subseteq Y$ and 
a rational-valued approximate isometry $\rho \in \operatorname{Apx}(X_0, Y_0)$ such that 
the relation $\psi \vartriangleleft \rho|^{X \times Y} \vartriangleleft \phi$ holds.  
\end{prop}

\begin{proof}
First, suppose (i) holds.  
Then there exist finite subsets $X_0 \subseteq X$ and $Y_0 \subseteq Y$ and 
a positive real number $\varepsilon$ such that the open neighborhood
\[
	U := \bigl\{ \psi' \in \operatorname{Apx}(X, Y) \bigm|
	|\psi'(x, y) - \psi(x, y)| < 2\varepsilon \text{ for any } x \in X_0, \ y \in Y_0 \bigr\}
\]
is included in $\operatorname{Apx}^{\leq\varphi}(X, Y)$.  
Clearly $(\psi|_{X_0 \times Y_0})|^{X \times Y} + \varepsilon$ belongs to $U$, 
so (iii) follows.  

It is trivial that (iii) implies (ii).  Now assume (ii).  
Since $\vartriangleleft$ is preserved under trivial extensions, 
the relation $\psi|_{X_0 \times Y_0} \vartriangleleft \psi|_{X_0 \times Y_0} + \varepsilon$ 
implies 
\[
	\psi \leq (\psi|_{X_0 \times Y_0})|^{X \times Y} 
	\vartriangleleft (\psi|_{X_0 \times Y_0})|^{X \times Y} + \varepsilon \leq \varphi, 
\]
so (i) holds.  

Finally, in order to find $\rho$ as in the statement, 
suppose $\psi \vartriangleleft \varphi$.  
Let $X_0, Y_0$ be as in the proof of (i) $\Rightarrow$ (iii) above, 
and $F_1, \dots, F_n$ be the partition of $X_0 \times Y_0$ induced by $\psi$.  
Without loss of generality, we may assume $\psi|_{F_1} < \dots < \psi|_{F_n}$.  
Take a function $\delta \colon X_0 \times Y_0 \to (0, \varepsilon)$ so that 
\begin{itemize}
\item
	$\delta$ is constant on each $F_i$; 
\item
	$\delta|_{F_n} < \psi|_{F_n} - \psi|_{F_{n-1}}$; 
\item
	$\delta|_{F_i} < \min\{\delta|_{F_{i+1}}, \psi|_{F_i} - \psi|_{F_{i-1}}\}$ 
	for $i = 2, \dots, n-1$; 
\item
	$\delta|_{F_1} < \min\{\delta|_{F_2}, \psi|_{F_1}\}$; and 
\item
	$\rho := \psi|_{X_0 \times Y_0} - \delta + \varepsilon$ is 
	a rational valued function on $X_0 \times Y_0$.  
\end{itemize}
We shall check that $\rho$ is separately Kat\v{e}tov so that it is an approximate isometry.  
The inequality 
\[
	d(x, x') \leq \rho(x, y) + \rho(x', y)
\]
is obvious, because $\rho \geq \psi|_{X_0 \times Y_0}$.  
On the other hand, for $(x, y) \in F_i$ and $(x', y) \in F_j$ with $i < j$, 
we have 
\begin{align*}
	\rho(x, y) 
	&= \psi|_{F_i} + \varepsilon - \delta|_{F_i} \\
	&= \psi|_{F_j} + \varepsilon - \delta|_{F_j} - 
	\bigl[\bigl((\psi|_{F_j} - \psi|_{F_i}) - \delta_{F_j}\bigr) + \delta|_{F_i}\bigr] \\
	&\leq \psi|_{F_j} + \varepsilon - \delta|_{F_j} = \rho(x', y) \\
	&\leq d(x, x') + \rho(x', y) 
\shortintertext{and}
	\rho(x', y)
	&= \psi(x', y) + \varepsilon - \delta|_{F_j} \\
	&\leq d(x', x) + \psi(x, y) + \varepsilon - \delta|_{F_i} \\
	&= d(x', x) + \rho(x, y), 
\end{align*}
so $\rho({}\cdot{}, y)$ is Kat\v{e}tov for each $y \in Y_0$.  
By symmetry, $\rho(x, {}\cdot{})$ is also Kat\v{e}tov, 
whence $\rho$ is an approximate isometry.  
Since clearly 
\[
	\psi|_{X_0 \times Y_0} \vartriangleleft \rho \vartriangleleft \varphi|_{X_0 \times Y_0}, 
\]
the conclusion follows.  
\end{proof}

\begin{lem}\label{lem:_perturbations_of_approximate_isometries}
Let $X, Y$ be metric spaces and 
$X_0, X_0' \subseteq X$ and $Y_0, Y_0' \subseteq Y$ be finite subsets.  
If $X_0'$ and $Y_0'$ are included in 
the $\varepsilon/5$-neighborhoods of $X_0$ and $Y_0$ respectively, 
then for any $\varphi \in \operatorname{Apx}(X, Y)$, the inequality 
\[
	(\varphi|_{X_0 \times Y_0})|^{X \times Y} + \varepsilon/5 
	\leq (\varphi|_{X_0' \times Y_0'})|^{X \times Y} + \varepsilon
\]
holds.  
\end{lem}

\begin{proof}
For $x' \in X_0'$ and $y' \in Y_0'$, there exist $x \in X_0$ and $y \in Y_0$ with 
$d(x', x) < \varepsilon/5$ and $d(y', y) < \varepsilon/5$, 
so we have 
\[
\begin{aligned}
	(\varphi|_{X_0 \times Y_0})|^{X \times Y} (x', y') + \varepsilon/5 
	&\leq d(x', x) + (\varphi|_{X_0 \times Y_0})|^{X \times Y}(x, y) + d(y, y') + \varepsilon/5 \\
	&= d(x', x) + \varphi(x, y) + d(y, y') + \varepsilon/5 \\
	&\leq 2d(x', x)+ \varphi(x', y') + 2d(y, y') + \varepsilon/5 \\
	&\leq \varphi(x', y') + \varepsilon.  
\end{aligned}
\]
\end{proof}


\section{Metric structures and approximate isomorphisms}
In this section, we recall the definition of metric structures and 
investigate fundamental properties of approximate isomorphisms.  
All the discussions are parallel to~\cite{yaacov15:_fraisse_limits}.  

By definition, a \emph{language} is a set $L$ such that 
each element of $L$ is either a \emph{function symbol} or a \emph{relation symbol}.  
To each symbol $S$ is associated a natural number $n_S$, which is called the \emph{arity} of $S$, 
and a symbol with arity $n$ is called an $n$-ary symbol.  
A $0$-ary function symbol is often called a \emph{constant symbol}.  

An \emph{$L$-structure} $\mathcal{M}$ is a complete metric structure $M$, 
which is called the \emph{domain} of $\mathcal{M}$, 
together with an \emph{interpretation} of symbols of $L$: 
\begin{itemize}
\item
	to each $n$-ary relation symbol $R$ is assigned a continuous map 
	$R^\mathcal{M}$ from $M^n$ to $\mathbb{R}$; and 
\item
	to each $n$-ary function symbol $f$ is assigned a continuous map $f^\mathcal{M}$ from 
	$M^n$ to $M$.  
\end{itemize}
For an $L$-structure $\mathcal{M}$, we shall denote its domain by $|\mathcal{M}|$.  

An \emph{$L$-embedding} of an $L$-structure $\mathcal{N}$ into 
another $L$-structure $\mathcal{M}$ is an isometry $\iota$
from $|\mathcal{N}|$ into $|\mathcal{M}|$ 
such that 
\begin{itemize}
\item
	for any $n$-ary relation symbol $R$ and any elements $a_1, \dots, a_n \in |\mathcal{N}|$, 
	the equation 
	\[
		R^\mathcal{N}(a_1, \dots, a_n) = R^\mathcal{M}\bigl(\iota(a_1), \dots, \iota(a_n)\bigr)
	\]
	holds, and
\item
	for any $n$-ary function symbol $f$ and any elements $a_1, \dots, a_n \in |\mathcal{N}|$, 
	the equation 
	\[
		\iota\bigl(f^\mathcal{N}(a_1, \dots, a_n)\bigr) 
		= f^\mathcal{M}\bigl(\iota(a_1), \dots, \iota(a_n)\bigr)
	\]
	holds.  
\end{itemize}
For an $L$-embedding $\iota \colon \mathcal{N} \to \mathcal{M}$ and 
a tuple $\bar{a} = (a_1, \dots, a_n) \in |\mathcal{N}|^n$, 
we shall write  the tuple $\bigl(\iota(a_1), \dots, \iota(a_n)\bigr) \in |\mathcal{M}|^n$ 
as $\iota(\bar{a})$.  

For a subset $E$ of an $L$-structure $\mathcal{M}$, 
the $L$-substructure generated by $E$ is denoted by $\langle E \rangle$.  
If it coincides with $\mathcal{M}$, 
then $E$ is said to be a \emph{generator} of $\mathcal{M}$.  
If $\mathcal{M}$ is generated by a finite subset, 
then $\mathcal{M}$ is said to be \emph{finitely generated}.  
A tuple $\bar{a} = (a_1, \dots, a_n) \in |\mathcal{M}|^n$ is 
an \emph{ordered generator} if $\{a_i \mid i = 1, \dots, n\}$ is a generator of $\mathcal{M}$.  

\medskip

In the sequel, we fix a language $L$ and a category $\mathscr{K}$ of 
finitely generated $L$-structures and $L$-embeddings.  
Embeddings and isomorphisms in $\operatorname{Mor}(\mathscr{K})$ are 
often referred to as \emph{$\mathscr{K}$-embeddings} and \emph{$\mathscr{K}$-isomorphisms} respectively.  
A \emph{joint $\mathscr{K}$-embedding} is a joint embedding $(\iota, \eta)$ such that 
both $\iota$ and $\eta$ are $\mathscr{K}$-embeddings.  
We denote by $\operatorname{JE}_\mathscr{K}(\mathcal{A}, \mathcal{B})$ 
the class of all joint $\mathscr{K}$-embeddings of $\mathcal{A}$ and $\mathcal{B}$.  

\begin{dfn}\label{dfn:_approximate_isomorphisms}
\begin{enumerate}
\item
	Let $\mathcal{A}, \mathcal{B}$ be objects of $\mathscr{K}$ and 
	$\iota \colon \mathcal{A} \dashrightarrow \mathcal{B}$ be a finite partial isometry, 
	that is, an isometry between finite subsets of $|\mathcal{A}|$ and $|\mathcal{B}|$.  
	Then $\iota$ is called a \emph{finite partial $\mathscr{K}$-isomorphism} if 
	\begin{itemize}
	\item
		the $L$-substructures $\langle \operatorname{dom} \iota \rangle$ and 
		$\langle \operatorname{ran} \iota \rangle$ are objects of $\mathscr{K}$; 
	\item
		the canonical embeddings 
		$\langle \operatorname{dom} \iota \rangle \to \mathcal{A}$ and 
		$\langle \operatorname{ran} \iota \rangle \to \mathcal{B}$ are 
		$\mathscr{K}$-embeddings; and 
	\item
		$\iota$ extends to a $\mathscr{K}$-isomorphism 
		from $\langle \operatorname{dom} \iota \rangle$ 
		onto $\langle \operatorname{ran} \iota \rangle$.  
	\end{itemize}
\item
	Let $\mathcal{A}, \mathcal{B}$ be objects of $\mathscr{K}$.  
	We denote by $\operatorname{Apx}_{2, \mathscr{K}}(\mathcal{A}, \mathcal{B})$ 
	the set of all approximate isometries from $|\mathcal{A}|$ to $|\mathcal{B}|$ 
	which are of the form $\varphi_{\iota, \eta}|^{\mathcal{A} \times \mathcal{B}}$, 
	where $\iota \colon \mathcal{A} \dashrightarrow \mathcal{C}$ and 
	$\eta \colon \mathcal{B} \dashrightarrow \mathcal{C}$ are 
	finite partial $\mathscr{K}$-isomorphisms into 
	some object $\mathcal{C}$ of $\mathscr{K}$.  
\item
	A \emph{$\mathscr{K}$-structure} is an $L$-structure $\mathcal{M}$ together with 
	an inductive system of $\mathscr{K}$-embeddings 
	\[
	\xymatrix{
		\mathcal{A}_1 \ar[r]^{\iota_1} & \mathcal{A}_2 \ar[r]^{\iota_2} &
		\mathcal{A}_3 \ar[r]^{\iota_3} & \cdots	
	}
	\]
	such that the inductive limit of the system as an $L$-structure is $\mathcal{M}$.  
	We often write $\mathcal{M} = \overline{\bigcup_n \mathcal{A}_n}$, 
	identifying each $\mathcal{A}_n$ as the corresponding $L$-substructure of $\mathcal{M}$.  
	Note that $\mathcal{M}$ is not necessarily an object of $\mathscr{K}$.  
\item
	For $\mathscr{K}$-structures $\mathcal{M} = \overline{\bigcup_n \mathcal{A}_n}$ and 
	$\mathcal{N} = \overline{\bigcup_m \mathcal{B}_m}$, we define 
	\[
		\operatorname{Apx}_\mathscr{K}(\mathcal{M}, \mathcal{N}) := 
		\operatorname{cl}\Bigl(\bigcup_{n, m} \bigl\{ \psi \bigm| \exists \varphi 
		\in \operatorname{Apx}_{2, \mathscr{K}}(\mathcal{A}_n, \mathcal{B}_m), \ 
		\psi \geq \varphi|^{\mathcal{M} \times \mathcal{N}}\bigr\}\Bigr)
	\]
	and call its elements \emph{approximate $\mathscr{K}$-isomorphisms}.  
	Also, we set 
	\[
	\begin{aligned}
		\operatorname{Apx}_\mathscr{K}^{\leq\varphi}(\mathcal{M}, \mathcal{N})
		&:= \operatorname{Apx}_\mathscr{K}(\mathcal{M}, \mathcal{N}) \cap 
		\operatorname{Apx}^{\leq\varphi}(|\mathcal{M}|, |\mathcal{N}|), \\
		\operatorname{Apx}_\mathscr{K}^{\vartriangleleft\varphi}(\mathcal{M}, \mathcal{N})
		&:= \operatorname{Apx}_\mathscr{K}(\mathcal{M}, \mathcal{N}) \cap 
		\operatorname{Apx}^{\vartriangleleft\varphi}(|\mathcal{M}|, |\mathcal{N}|).  
	\end{aligned}
	\]
	If $\operatorname{Apx}_\mathscr{K}^{\vartriangleleft\varphi}(\mathcal{M}, \mathcal{N})$ is 
	nonempty, then $\varphi$ is said to be \emph{strict}.  
	We denote the set of strict approximate $\mathscr{K}$-isomorphisms 
	from $\mathcal{M}$ to $\mathcal{N}$ by $\operatorname{Stx}_\mathscr{K}(\mathcal{M}, \mathcal{N})$.  
\item
	An $L$-embedding $\iota$ of a $\mathscr{K}$-structure 
	$\mathcal{M} = \overline{\bigcup_n \mathcal{A}_n}$ into 
	another $\mathscr{K}$-structure $\mathcal{N} = \overline{\bigcup_m \mathcal{B}_n}$ is 
	said to be \emph{$\mathscr{K}$-admissible} 
	if the corresponding approximate isometry $\varphi_\iota$ belongs to 
	$\operatorname{Apx}_\mathscr{K}(\mathcal{M}, \mathcal{N})$.  
	Two $\mathscr{K}$-structures are understood to be isomorphic 
	if there exists a $\mathscr{K}$-admissible isomorphism between them.  
\end{enumerate}
\end{dfn}

An object $\mathcal{A}$ of $\mathscr{K}$ can be canonically identified with 
a $\mathscr{K}$-structure obtained from the inductive system 
\[
\xymatrix{
	\mathcal{A} \ar[r]^{\operatorname{id}} & 
	\mathcal{A} \ar[r]^{\operatorname{id}} & \cdots,  
}
\]
so that we can consider $\operatorname{Apx}_\mathscr{K}(\mathcal{A},\mathcal{B})$ 
for objects $\mathcal{A}, \mathcal{B}$ of $\mathscr{K}$.  
If $\mathcal{A}, \mathcal{B}, \mathcal{C}$ are objects of $\mathscr{K}$ and 
$\iota \colon \mathcal{A} \to \mathcal{C}$ and $\eta \colon \mathcal{B} \to \mathcal{C}$ are 
$\mathscr{K}$-embeddings, then $\varphi_{\iota, \eta}$ belongs to 
$\operatorname{Apx}_\mathscr{K}(\mathcal{A},\mathcal{B})$, 
because it is the limit of 
\[
	\bigl\{ (\varphi_{\iota,\eta}|_{A_0 \times B_0})|^{\mathcal{A} \times \mathcal{B}} 
	\bigm| A_0 \subseteq |\mathcal{A}|, B_0 \subseteq |\mathcal{B}| 
	\text{ are finite generators} \bigr\} \subseteq 
	\operatorname{Apx}_{2, \mathscr{K}}(\mathcal{A},\mathcal{B}).  
\]
In particular, every $\mathscr{K}$-embedding is $\mathscr{K}$-admissible.  
On the other hand, note that there might be a $\mathscr{K}$-admissible isomorphism 
between objects of $\mathscr{K}$ which is not a morphism of $\mathscr{K}$. 
There can be even a $\mathscr{K}$-admissible $\iota \colon \mathcal{A} \to \mathcal{B}$ such that 
no net of $\mathscr{K}$-embeddings of $\mathcal{A}$ into $\mathcal{B}$ 
converges to $\iota$.  

For any approximate $\mathscr{K}$-isomorphism $\varphi$ from $\mathcal{M}$ to $\mathcal{N}$, 
the set $\operatorname{Apx}_\mathscr{K}^{\vartriangleleft\varphi}(\mathcal{M},\mathcal{N})$ is 
obviously included in the relative interior of 
$\operatorname{Apx}_\mathscr{K}^{\leq\varphi}(\mathcal{M},\mathcal{N})$ in 
$\operatorname{Apx}_\mathscr{K}(\mathcal{M},\mathcal{N})$.  
The opposite inclusion also holds, 
because any relative interior point $\psi$ in 
$\operatorname{Apx}_\mathscr{K}^{\leq\varphi}(\mathcal{M},\mathcal{N})$ satisfies~(ii) 
in Proposition~\ref{prop:_characterization_of_strict_domination}.  

Given a subset $A$ of $\operatorname{Apx}(X, Y)$, we shall define 
\[
	A^\uparrow := \{\psi \in \operatorname{Apx}(X, Y) \mid 
	\exists \varphi \in A, \ \psi \geq \varphi \}.  
\]
Then it can be shown that $\operatorname{cl}(A^\uparrow)$ is still upward closed, 
that is, $\operatorname{cl}(A^\uparrow)^\uparrow = \operatorname{cl}(A^\uparrow)$.  
Indeed, if $\varphi \geq \varphi'$ for $\varphi' \in \operatorname{cl}(A^\uparrow)$, 
then for any $\varepsilon > 0$ and any finite subsets $X_0 \subseteq X$ and $Y_0 \subseteq Y$, 
we can find an approximate isometry $\psi \in A^\uparrow$ such that 
the inequality $|\psi - \varphi'| < \varepsilon$ on $X_0 \times Y_0$ holds.  
It follows that $(\varphi|_{X_0 \times Y_0})^{X \times Y} + \varepsilon$ is in $A^\uparrow$, 
so $\varphi$ is in $\operatorname{cl}(A^\uparrow)$.  
In particular, 
$\operatorname{Apx}_\mathscr{K}(\mathcal{M}, \mathcal{N})$ is 
upward closed for any $\mathscr{K}$-structures $\mathcal{M} = \overline{\bigcup_n \mathcal{A}_n}$ 
and $\mathcal{N} = \overline{\bigcup_m \mathcal{B}_m}$. 
This argument also implies that $\operatorname{Stx}_\mathscr{K}(\mathcal{M}, \mathcal{N})$ is 
topologically dense in $\operatorname{Apx}_\mathscr{K}(\mathcal{M}, \mathcal{N})$, 
since for any approximate $\mathscr{K}$-isomorphism $\varphi$, 
it automatically follows that the approximate isometries of the form 
$(\varphi|_{M_0 \times N_0})^{\mathcal{M} \times \mathcal{N}} + \varepsilon$ 
are indeed strict approximate isomorphisms, 
where $M_0 \subseteq |\mathcal{M}|$ and $N_0 \subseteq |\mathcal{N}|$ 
are arbitrary finite subsets and $\varepsilon$ is any positive real number.  

\begin{dfn}\label{dfn:_jep_and_nap}
The category $\mathscr{K}$ is said to satisfy 
\begin{itemize}
\item
	the \emph{joint embedding property} (JEP) if 
	$\operatorname{JE}_\mathscr{K}(\mathcal{A}, \mathcal{B})$ is nonempty for 
	any objects $\mathcal{A}, \mathcal{B}$ of $\mathscr{K}$.  
\item
	the \emph{near amalgamation property} (NAP) if 
	for any objects $\mathcal{A}, \mathcal{B}_1, \mathcal{B}_2$ in $\mathscr{K}$, 
	any $\mathscr{K}$-embeddings $\iota_i \colon \mathcal{A} \to \mathcal{B}_i$, 
	any finite subset $F \subseteq |\mathcal{A}|$ and any $\varepsilon > 0$, 
	there exists a joint $\mathscr{K}$-embedding $(\eta_1, \eta_2)$ of $\mathcal{B}_1$ and 
	$\mathcal{B}_2$ such that the inequality 
	\[
		d\bigl(\eta_1 \circ \iota_1(a), \kern 3pt \eta_2 \circ \iota_2(a)\bigr) < \varepsilon
	\]
	holds for all $a \in F$.  
\end{itemize}
\end{dfn}

The following propositions are essential in 
proving the existence and uniqueness of Fraïssé limits in the next section.  
In fact, the gap of the theory presented in~\cite{eagle16:_fraisse_limits} is also related to 
these propositions, as is explained in Remark~\ref{rem:_the_gap}.  

\begin{prop}\label{prop:_joint_embedding_and_approximate_isomorphism}
Suppose that $\mathscr{K}$ satisfies NAP.  
Then for any objects $\mathcal{A}, \mathcal{B}$ of $\mathscr{K}$ and 
any strict approximate $\mathscr{K}$-isomorphism $\varphi$ from $\mathcal{A}$ to $\mathcal{B}$, 
there exists a joint $\mathscr{K}$-embedding $(\iota, \eta)$ of $\mathcal{A}$ and $\mathcal{B}$ 
satisfying $\varphi_{\iota, \eta} \vartriangleleft \varphi$.  
\end{prop}

\begin{proof}
Since $\varphi$ is strict, 
$\operatorname{Apx}_\mathscr{K}^{\vartriangleleft\varphi}(\mathcal{A}, \mathcal{B})$ is 
an open nonempty subset of 
$\operatorname{Apx}_\mathscr{K}(\mathcal{A}, \mathcal{B})$.        
Therefore, it intersects with $\operatorname{Apx}_{2, \mathscr{K}}(\mathcal{A}, \mathcal{B})$, 
as $\operatorname{Apx}_{2, \mathscr{K}}(\mathcal{A}, \mathcal{B})^\uparrow$ is a dense subset.  
In other words, there exist an object $\mathcal{C}_0$ of $\mathscr{K}$ and 
finite partial $\mathscr{K}$-isomorphisms 
$\iota_0 \colon \mathcal{A} \dashrightarrow \mathcal{C}_0$ and 
$\eta_0 \colon \mathcal{B} \dashrightarrow \mathcal{C}_0$ such that the relation 
$\varphi_{\iota_0, \eta_0}|^{\mathcal{A} \times \mathcal{B}} + \varepsilon 
\vartriangleleft \varphi$ holds 
for some $\varepsilon > 0$.  
Put $\mathcal{A}_0 := \langle \operatorname{dom} \iota_0 \rangle$ and 
$\mathcal{B}_0 := \langle \operatorname{dom} \eta_0 \rangle$.  
By the definition of finite partial $\mathscr{K}$-isomorphisms, 
the canonical embeddings $\mathcal{A}_0 \to \mathcal{A}$ and 
$\mathcal{B}_0 \to \mathcal{B}$ are $\mathscr{K}$-embeddings.  

Now, by NAP there exist $\mathscr{K}$-embeddings
$\iota_1 \colon \mathcal{A} \to \mathcal{C}_\mathcal{A}$ and 
$\iota_1' \colon \mathcal{C}_0 \to \mathcal{C}_\mathcal{A}$ such that 
the inequality $d\bigl(\iota_1(a), \kern 3pt \iota_1' \circ \iota_0(a)\bigr) < \varepsilon/3$ holds 
for all $a \in \operatorname{dom} \iota_0$.  
Similarly, there are $\mathscr{K}$-embeddings $\eta_1 \colon \mathcal{B} \to \mathcal{C}_\mathcal{B}$ 
and $\eta_1' \colon \mathcal{C}_0 \to \mathcal{C}_\mathcal{B}$ with 
$d\bigl(\eta_1(b), \kern 3pt \eta_1' \circ \eta_(b)\bigr) < \varepsilon/3$ 
for all $b \in \operatorname{dom} \eta_0$.  
Then, again by NAP, there exist $\mathscr{K}$-embeddings 
$\iota_2 \colon \mathcal{C}_\mathcal{A} \to \mathcal{C}$ and 
$\eta_2 \colon \mathcal{C}_\mathcal{B} \to \mathcal{C}$ with 
$d\bigl(\iota_2 \circ \iota_1'(c), \kern 3pt \eta_2 \circ \eta_1'(c)\bigr) < \varepsilon/3$ 
for any $c \in \operatorname{ran} \iota_0 \cup \operatorname{ran} \iota_0$.  
\[
\xymatrix{
	& \mathcal{A}_0 \ar@{^{(}->}[r] \ar@{-->}[d]^{\iota_0} 
		& \mathcal{A} \ar[d]^{\iota_1} \\
	\mathcal{B}_0 \ar@{-->}[r]^{\eta_0} \ar@{_{(}->}[d] 
		& \mathcal{C}_0 \ar[r]^{\iota_1'} \ar[d]^{\eta_1'} 
		& \mathcal{C}_\mathcal{A} \ar[d]^{\iota_2} \\
	\mathcal{B} \ar[r]^{\eta_1} & \mathcal{C}_\mathcal{B} \ar[r]^{\eta_2} & \mathcal{C}
}
\]

Set $\iota := \iota_2 \circ \iota_1$ and $\eta := \eta_2 \circ \eta_1$.  
Then for $a \in \operatorname{dom} \iota_0$ and $b \in \operatorname{dom} \eta_0$, we have 
\[
\begin{aligned}
	d\bigl(\iota(a), \eta(b)\bigr) 
	&= d\bigl(\iota_2 \circ \iota_1(a), \kern 3pt \eta_2 \circ \eta_1(b)\bigr) \\
	&\leq d\bigl(\iota_2 \circ \iota_1' \circ \iota_0(a), \kern 3pt 
	\eta_2 \circ \eta_1' \circ \eta_0(b)\bigr) + 2\varepsilon/3 \\
	&\leq d\bigl(\iota_2 \circ \iota_1' \circ \iota_0(a), \kern 3pt 
	\iota_2 \circ \iota_1' \circ \eta_0(b)\bigr) + \varepsilon 
	= d\bigl(\iota_0(a), \eta_0(b)\bigr) + \varepsilon, \\	
\end{aligned}
\]
so $\varphi_{\iota, \eta} \leq 
\varphi_{\iota_0, \eta_0}|^{\mathcal{A} \times \mathcal{B}} + \varepsilon 
\vartriangleleft \varphi$, as desired.  
\end{proof}

\begin{prop}\label{prop:_compositions_of_approximate_isomorphisms}
Let $\mathcal{M}_1 = \overline{\bigcup_l \mathcal{A}_l}$, 
$\mathcal{M}_2 = \overline{\bigcup_m \mathcal{B}_m}$ and 
$\mathcal{M}_3 = \overline{\bigcup_n \mathcal{C}_n}$ be 
$\mathscr{K}$-structures.  
If $\varphi$ and $\psi$ belongs to $\operatorname{Apx}_\mathscr{K}(\mathcal{M}_1, \mathcal{M}_2)$ and 
$\operatorname{Apx}_\mathscr{K}(\mathcal{M}_2, \mathcal{M}_3)$ respectively, 
then the composition $\psi\varphi$ is 
in $\operatorname{Apx}_\mathscr{K}(\mathcal{M}_1, \mathcal{M}_3)$.  
\end{prop}

\begin{proof}
First, assume that both $\varphi$ and $\psi$ are strict 
and $\mathcal{M}_1$, $\mathcal{M}_2$, $\mathcal{M}_3$ are objects of $\mathscr{K}$.  
Then, by Proposition~\ref{prop:_joint_embedding_and_approximate_isomorphism}, 
there exist objects $\mathcal{D}$ and $\mathcal{E}$ of $\mathscr{K}$ and 
$\mathscr{K}$-embeddings $\iota_i \colon \mathcal{M}_i \to \mathcal{D} \ (i = 1, 2)$ and
$\eta_j \colon \mathcal{M}_j \to \mathcal{E} \ (j = 2, 3)$ such that 
$\varphi_{\iota_1, \iota_2} \vartriangleleft \varphi$ and 
$\varphi_{\eta_2, \eta_3} \vartriangleleft \psi$.  
It follows from Proposition~\ref{prop:_characterization_of_strict_domination} that 
there exist a finite subset $F_0 \subseteq |\mathcal{M}_2|$ and 
a positive real number $\varepsilon > 0$ with 
\[
\begin{aligned}
(\varphi_{\iota_1, \iota_2}|_{\mathcal{M}_1 \times F_0})|^{\mathcal{M}_1 \times \mathcal{M}_2} 
+ \varepsilon &\vartriangleleft \varphi, & 
(\varphi_{\eta_2, \eta_3}|_{F_0 \times \mathcal{M}_3})|^{\mathcal{M}_2 \times \mathcal{M}_3} 
+ \varepsilon &\vartriangleleft \psi.  
\end{aligned}
\]
By NAP, we can find $\mathscr{K}$-embeddings $\theta_1 \colon \mathcal{D} \to \mathcal{F}$ and 
$\theta_2 \colon \mathcal{E} \to \mathcal{F}$ such that 
the inequality $d\bigl(\theta_1 \circ \iota_2(b), \kern 3pt 
\theta_2 \circ \eta_2(b)\bigr) < 2\varepsilon$ holds 
for all $b \in F_0$.  
\[
\xymatrix{
	 & & \mathcal{M}_1 \ar[d]^{\iota_1} \\
	 & \mathcal{M}_2 \ar[r]^{\iota_2} \ar[d]^{\eta_2} & \mathcal{D} \ar[d]^{\theta_1} \\
	\mathcal{M}_3 \ar[r]^{\eta_3} & \mathcal{E} \ar[r]^{\theta_2} & \mathcal{F}
}
\]
For $a \in |\mathcal{M}_1|$ and $c \in |\mathcal{M}_3|$, we have 
\[
\begin{aligned}
	&d\bigl(\theta_1 \circ \iota_1(a), \kern 3pt \theta_2 \circ \eta_3(c)\bigr) \\
	\leq& \inf_{b \in F_0} \Bigl[ d\bigl(\theta_1 \circ \iota_1(a), \kern 3pt 
	\theta_1 \circ \iota_2(b)\bigr) + 
	d\bigl(\theta_1 \circ \iota_2(b), \kern 3pt \theta_2 \circ \eta_3(c)\bigr) \Bigr] \\
	<& \inf_{b \in F_0} \Bigl[ d\bigl(\theta_1 \circ \iota_1(a), \kern 3pt 
	\theta_1 \circ \iota_2(b)\bigr) + 
	d\bigl(\theta_2 \circ \eta_2(b), \kern 3pt \theta_2 \circ \eta_3(c)\bigr) 
	+ 2\varepsilon \Bigr] \\
	=& \bigl(\varphi_{\eta_2, \eta_3}|_{F_0 \times \mathcal{M}_3} + \varepsilon\bigr)
	\bigl(\varphi_{\iota_1, \iota_2}|_{\mathcal{M}_1 \times F_0} + \varepsilon\bigr)(a, c), 
\end{aligned}
\]
so 
\[
	\varphi_{\theta_1 \circ \iota_1, \theta_2 \circ \eta_3} 
	\leq \Bigl[\bigl(\varphi_{\eta_2, \eta_3}|_{F_0 \times \mathcal{M}_3} + 
	\varepsilon\bigr)|^{\mathcal{M}_2 \times \mathcal{M}_3}\Bigr]
	\Bigl[\bigl(\varphi_{\iota_1, \iota_2}|_{\mathcal{M}_1 \times F_0} + 
	\varepsilon\bigr)|^{\mathcal{M}_1 \times \mathcal{M}_2}\Bigr] 
	\leq \psi\varphi.  
\]
Since $\varphi_{\theta_1 \circ \iota_1, \theta_2 \circ \eta_3}$ is in 
$\operatorname{Apx}_\mathscr{K}(\mathcal{M}_1, \mathcal{M}_3)$, so is $\psi\varphi$.  

Next, assume that both $\varphi$ and $\psi$ are still strict, 
but $\mathcal{M}_1$, $\mathcal{M}_2$ and $\mathcal{M}_3$ are 
general $\mathscr{K}$-structures.  
Then there exist sufficiently large $l, m, m', n \in \mathbb{N}$ and 
approximate $\mathscr{K}$-isomorphisms 
$\varphi'$ from $\mathcal{A}_l$ to $\mathcal{B}_m$ and 
$\psi'$ from $\mathcal{B}_{m'}$ to $\mathcal{C}_n$ with
$\varphi'|^{\mathcal{M}_1 \times \mathcal{M}_2} \vartriangleleft \varphi$ and 
$\psi'|^{\mathcal{M}_2 \times \mathcal{M}_3} \vartriangleleft \psi$.  
We may assume without loss of generality that 
$m$ is equal to $m'$, since in general, 
if $\iota \colon \mathcal{A} \to \mathcal{A}'$ and $\eta \colon \mathcal{B} \to \mathcal{B}'$ 
are $\mathscr{K}$-embeddings, then one can directly check from the definition that 
the trivial extension of an approximate $\mathscr{K}$-isomorphism in 
$\operatorname{Apx}_{2, \mathscr{K}}(\mathcal{A}, \mathcal{B})$ 
via these $\mathscr{K}$-embeddings belong to 
$\operatorname{Apx}_{2, \mathscr{K}}(\mathcal{A}', \mathcal{B}')$.  
Also, we may assume that both $\varphi'$ and $\psi'$ are strict 
by Proposition~\ref{prop:_characterization_of_strict_domination}.  
By what we proved in the preceding paragraph, $\psi'\varphi'$ is 
in $\operatorname{Apx}_\mathscr{K}(\mathcal{A}_l, \mathcal{C}_n)$.  
By direct computation, one can check that 
\[
	(\psi'|^{\mathcal{M}_2 \times \mathcal{M}_3})
	(\varphi'|^{\mathcal{M}_1 \times \mathcal{M}_2}) 
	= (\psi'\varphi)|^{\mathcal{M}_1 \times \mathcal{M}_3}, 
\]
so $\psi\varphi$ is in $\operatorname{Apx}_\mathscr{K}(\mathcal{M}_1, \mathcal{M}_3)$.  

Finally, let $\varphi$ and $\psi$ be general approximate $\mathscr{K}$-isomorphisms 
between general $\mathscr{K}$-structures.  
Then there exist nets $\{\varphi_\alpha\}$ and $\{\psi_\beta\}$ of 
strict approximate $\mathscr{K}$-isomorphisms which converge to $\varphi$ and $\psi$ respectively, 
and 
\[
	\psi\varphi = (\varlimsup_\beta \psi_\beta)(\varlimsup_\alpha \varphi_\alpha) 
	\geq \varlimsup_{\alpha, \beta}(\psi_\beta\varphi_\alpha) 
	\in \operatorname{Apx}_\mathscr{K}(\mathcal{M}_1, \mathcal{M}_3), 
\]
so $\psi\varphi$ belongs to 
$\operatorname{Apx}_\mathscr{K}(\mathcal{M}_1, \mathcal{M}_3)$.  
\end{proof}

\begin{cor}\label{cor:_trivial_extensions_and_restrictions}
Extensions and restrictions of approximate $\mathscr{K}$-isomorphisms via 
$\mathscr{K}$-admissible embeddings are approximate $\mathscr{K}$-isomorphisms.  
\end{cor}

\section{Fraïssé categories and their limits}
Let $L$ be a language and $\mathscr{K}$ be a category of 
finitely generated $L$-structures and $L$-embeddings with JEP and NAP.  
For each $n \in \mathbb{N}$, we denote by $\mathscr{K}_n$ 
the class of all pairs $\langle \mathcal{A}, \bar{a} \rangle$, 
where $\mathcal{A}$ is an object of $\mathscr{K}$ and 
$\bar{a}$ is an ordered generator of $\mathcal{A}$.  
We simply write $\langle \bar{a} \rangle$ instead of 
$\langle \mathcal{A}, \bar{a} \rangle$ when there is no danger of confusion.  

For each $n$, we consider a pseudo-metric on $\mathscr{K}_n$ defined by 
\[
\begin{aligned}
	d^\mathscr{K}\bigl(\langle \bar{a} \rangle, \langle \bar{b} \rangle\bigr) 
	:=& \inf \bigl\{ \max_i \varphi(a_i, b_i) \bigm| 
	\varphi \in \operatorname{Apx}_\mathscr{K}
	\bigl(\langle \bar{a} \rangle, \langle \bar{b} \rangle\bigr) \bigr\} \\
	=& \inf \bigl\{ \max_i \varphi(a_i, b_i) \bigm| 
	\varphi \in \operatorname{Stx}_\mathscr{K}
	\bigl(\langle \bar{a} \rangle, \langle \bar{b} \rangle\bigr) \bigr\} \\
	=& \inf \bigl\{ \max_i d\bigl(\iota(a_i), \eta(b_i)\bigr) \bigm| 
	(\iota, \eta) \in \operatorname{JE}_\mathscr{K}
	\bigl(\langle \bar{a} \rangle, \langle \bar{b} \rangle\bigr) \bigr\}, 
\end{aligned}
\]
where $a_i$ and $b_i$ denotes the $i$-th component of $\bar{a}$ and $\bar{b}$ respectively.  
The fact that $d^\mathscr{K}$ is indeed a pseudo-metric follows from JEP and NAP.  

\begin{dfn}\label{dfn:_wpp_and_ccp}
The category $\mathscr{K}$ is said to satisfy 
\begin{itemize}
\item
	the \emph{weak Polish property} (WPP) if 
	$\mathscr{K}_n$ is separable with respect to the pseudo-metric $d^\mathscr{K}$ for each $n$.  
\item
	the \emph{Cauchy continuity property} (CCP) if 
	\begin{enumerate}[label=(\roman*)]
	\item
		for any $n$-ary predicate symbol $P$ in $L$, the map 
		\[
			\bigl\langle \mathcal{A}, (\bar{a}, \bar{b}) \bigr\rangle \mapsto P^\mathcal{A}(\bar{a})
		\]
		from $\mathscr{K}_{n+m}$ into $\mathbb{R}$ sends 
		Cauchy sequences into Cauchy sequences; and 
	\item
		for any $n$-ary function symbol $f$ in $L$, the map 
		\[
			\bigl\langle \mathcal{A}, (\bar{a}, \bar{b}) \bigr\rangle \mapsto 
			\bigl\langle \mathcal{A}, 
			\bigl(\bar{a}, \bar{b}, f^\mathcal{A}(\bar{a})\bigr) \bigr\rangle 
		\]
		from $\mathscr{K}_{n+m}$ into $\mathscr{K}_{n+m+1}$ sends 
		Cauchy sequences into Cauchy sequences.  
	\end{enumerate}
\end{itemize}
\end{dfn}

\begin{rem}\label{rem:_effect_of_ccp}
If $\mathscr{K}$ satisfies CCP, 
then $d^\mathscr{K}(\langle \bar{a} \rangle, \langle \bar{b} \rangle)$ is equal to zero 
if and only if there exists a $\mathscr{K}$-admissible isomorphism from 
$\langle \bar{a} \rangle$ onto $\langle \bar{b} \rangle$ which sends $a_i$ to $b_i$.  
To see this, first suppose that the map $a_i \mapsto b_i$ extends to 
a $\mathscr{K}$-admissible isomorphism $\iota$.  
Then $(\varphi_\iota|_{\bar{a} \times \bar{b}})
|^{\langle \bar{a} \rangle \times \langle \bar{b} \rangle} + \varepsilon$ belongs to 
$\operatorname{Stx}_\mathscr{K}\bigl(\langle \bar{a} \rangle, \langle \bar{b} \rangle\bigr)$, and 
\[
	d^\mathscr{K}\bigl(\langle \bar{a} \rangle, \langle \bar{b} \rangle\bigr) 
	\leq (\varphi_\iota|_{\bar{a} \times \bar{b}})
	|^{\langle \bar{a} \rangle \times \langle \bar{b} \rangle}(a_i, b_i) + \varepsilon 
	= \varepsilon
\]
for arbitrary $\varepsilon > 0$.  
Conversely, suppose $d^\mathscr{K}\bigl(\langle \bar{a} \rangle, \langle \bar{b} \rangle\bigr) = 0$.  
Let $D_{\bar{a}}$ be the set of all elements of $\langle \bar{a} \rangle$ of the form 
$g(\bar{a})$, where $g$ is a composition of functions equipped with $\langle \bar{a} \rangle$, 
and $D_{\bar{b}}$ be the set obtained from $\langle \bar{b} \rangle$ by the same way.  
Then it follows from CCP that the map $a_i \mapsto b_i$ extends to 
an isometry from $D_{\bar{a}}$ onto $D_{\bar{b}}$ and 
the interpretations of the symbols can be identified via this isometry, 
so that it extends to an $L$-isomorphism $\iota$ from $\langle \bar{a} \rangle$ onto 
$\langle \bar{b} \rangle$.  
If $\bar{c} = \bigl(g_1(\bar{a}), \dots, g_n(\bar{a})\bigr)$ and 
$\bar{d} = \bigl(g_1(\bar{b}), \dots, g_n(\bar{b})\bigr)$, then 
$d^\mathscr{K}\bigl(\langle \bar{a}, \bar{c} \rangle, \langle \bar{b}, \bar{d} \rangle\bigr) = 0$ 
by CCP, so there exists a joint $\mathscr{K}$-embedding $(\eta_1, \eta_2)$ of 
$\langle \bar{a}, \bar{c} \rangle$ and $\langle \bar{b}, \bar{d} \rangle$ such that 
the $\eta_1(\bar{a}, \bar{c})$ and $\eta_2(\bar{b}, \bar{d})$ are arbitrarily close to each other, 
whence $\iota$ is $\mathscr{K}$-admissible.  
\end{rem}

\begin{dfn}\label{dfn:_fraisse_category}
\begin{enumerate}
\item
	A category $\mathscr{K}$ of finitely generated separable $L$-structures is 
	called a \emph{Fraïssé category} if it satisfies JEP, NAP, WPP and CCP. 
\item
	Let $\mathscr{K}$ be a Fraïssé class.   
	A $\mathscr{K}$-structure $\mathcal{M}$ is called a \emph{Fraïssé limit} of $\mathscr{K}$ 
	if for any $\mathscr{K}$-structure $\mathcal{N}$ and 
	any strict approximate $\mathscr{K}$-isomorphism 
	$\varphi \in \operatorname{Stx}_\mathscr{K}(\mathcal{N}, \mathcal{M})$, 
	there exists a $\mathscr{K}$-admissible embedding $\iota \colon \mathcal{N} \to \mathcal{M}$ with 
	$\varphi_\iota \vartriangleleft \varphi$.  
\end{enumerate}
\end{dfn}

We shall begin with characterizing Fraïssé limits.  
Fix a Fraïssé class $\mathscr{K}$.  

\begin{dfn}\label{dfn:_ultra-homogeneity_and_universality}
A $\mathscr{K}$-structure $\mathcal{M}$ is said to be 
\begin{itemize}
\item
	\emph{$\mathscr{K}$-universal} if for any object $\mathcal{A}$ of $\mathscr{K}$, 
	there exists a $\mathscr{K}$-admissible embedding of $\mathcal{A}$ into $\mathcal{M}$.  
\item
	\emph{approximately $\mathscr{K}$-ultra-homogeneous} if 
	for any $\langle \bar{a} \rangle \in \mathscr{K}_n$, any $\varepsilon > 0$ and 
	any $\mathscr{K}$-admissible embeddings 
	$\iota, \eta \colon \langle \bar{a} \rangle \to \mathcal{M}$, 
	there exists a $\mathscr{K}$-admissible automorphism $\alpha$ of $\mathcal{M}$ 
	with $\max_i d\bigl(\alpha \circ \iota(a_i), \kern 3pt \eta(a_i)\bigr) \leq \varepsilon$.  
\end{itemize}
\end{dfn}

\begin{thm}\label{thm:_uniqueness_of_fraisse_limit}
For a $\mathscr{K}$-structure $\mathcal{M}$, the following are equivalent.  
\begin{enumerate}[label=\textup{(\roman*)}]
\item
	The structure $\mathcal{M}$ is a Fraïssé limit of $\mathscr{K}$.  
\item
	For any object $\mathcal{A}$ of $\mathscr{K}$ and any 
	$\varphi \in \operatorname{Stx}_\mathscr{K}(\mathcal{A}, \mathcal{M})$, 
	there exists a $\mathscr{K}$-admissible embedding 
	$\iota \colon \mathcal{A} \to \mathcal{M}$ with $\varphi_\iota \vartriangleleft \varphi$.  
\item
	If $\langle \bar{a} \rangle$ is in $\mathscr{K}_n$ and $\varphi$ is 
	a strict approximate $\mathscr{K}$-isomorphism from $\langle \bar{a} \rangle$ to $\mathcal{M}$, 
	then for any $\varepsilon > 0$ there is an approximate $\mathscr{K}$-isomorphism 
	$\psi \in \operatorname{Stx}_\mathscr{K}^{\vartriangleleft\varphi}
	\bigl(\langle \bar{a} \rangle, \mathcal{M}\bigr)$ such that 
	$\psi|_{\bar{a} \times \mathcal{M}}$ is $\varepsilon$-total.  
\item
	The structure $\mathcal{M}$ is $\mathscr{K}$-universal and 
	approximately $\mathscr{K}$-ultra-homogeneous.  
\end{enumerate}
Moreover, if a Fraïssé limit exists, then it is unique up to $\mathscr{K}$-admissible isomorphisms.  
\end{thm}

\begin{proof}
First, assume that $\mathcal{M} = \overline{\bigcup_n \mathcal{A}_n}$ and 
$\mathcal{N} = \overline{\bigcup_m \mathcal{B}_m}$ are $\mathscr{K}$-structures satisfying~(iii).  
We shall show that if $\varphi$ is a strict approximate $\mathscr{K}$-isomorphism 
from $\mathcal{M}$ to $\mathcal{N}$, 
then there exists a $\mathscr{K}$-admissible isomorphism 
$\alpha$ from $\mathcal{M}$ onto $\mathcal{N}$ with $\varphi_\alpha \vartriangleleft \varphi$.  
Since $\varphi$ is strict, 
there exist an approximate $\mathscr{K}$-isomorphism $\psi$ from $\mathcal{M}$ to $\mathcal{N}$, 
finite subsets $E \subseteq \bigcup_n |\mathcal{A}_n|$ and 
$F \subseteq \bigcup_m |\mathcal{B}_m|$, and a positive real number $\varepsilon \leq 1$ 
with 
\[
	(\psi|_{E \times F})|^{\mathcal{M} \times \mathcal{N}} + \varepsilon 
	\vartriangleleft \varphi.  
\]
Take increasing sequences $\{X_i\}$ and $\{Y_j\}$ of finite sets such that 
\begin{itemize}
\item
	$X_0 = E$ and $Y_0 = F$; 
\item
	$X_i \subseteq \bigcup_n |\mathcal{A}_n|$ and $Y_j \subseteq \bigcup_m |\mathcal{B}_m|$ 
	for all $i, j$; and 
\item
	$\bigcup_i X_i$ and $\bigcup_j Y_j$ are dense in $|\mathcal{M}|$ and $|\mathcal{N}|$ 
	respectively.  
\end{itemize}
We claim the existence of a sequence $\{\psi_l\}$ of 
strict approximate $\mathscr{K}$-isomorphisms form $\mathcal{M}$ to $\mathcal{N}$ 
with the following properties.  
\begin{enumerate}[label=(\alph*)]
\item
	each $\psi_l$ is of the form $(\theta|_{X_{i(l)} \times Y_{j(l)}})
	|^{\mathcal{M} \times \mathcal{N}} + \varepsilon$ for some 
	$\theta \in \operatorname{Apx}_\mathscr{K}(\mathcal{M}, \mathcal{N})$, 
	where $\delta_l \leq 2^{-l}$ and $i(l), j(l) \uparrow \infty$; 
\item
	$\psi_{l+1} \vartriangleleft \psi_l$; and 
\item
	$\psi_{l+1}|_{X_{i(l)} \times \mathcal{N}}$ is $\delta_l$-total if $l$ is even, 
	while $\psi_{l+1}|_{\mathcal{M} \times Y_{i(l)}}$ is $\delta_l$-surjective if $l$ is odd.  
\end{enumerate}
The construction of such a sequence proceeds as following.  Set 
\[
	\psi_0 := (\psi|_{X_0 \times Y_0})|^{|\mathcal{M}| \times |\mathcal{N}|} + \varepsilon.  
\]
Assume $l$ is even and $\psi_l$ is given.  
Then, by assumption on $\mathcal{N}$, 
one can find $\theta \vartriangleleft \psi_l$ such that 
$\theta|_{X_{i(l)} \times \mathcal{N}}$ is $\delta_l/2$-total.  
Since $(\theta|_{X_i \times Y_j})|^{\mathcal{M} \times \mathcal{N}} + \delta$ 
converges to $\theta$ as $i, j \to \infty$ and $\delta \to 0$, 
for sufficiently large $i(l+1) > i(l)$ and $j(l+1) > j(l)$ and 
sufficiently small $\delta_{l+1} < \delta_l/2$, we have 
\[
	\psi_l \vartriangleright 
	(\theta|_{X_{i(l+1)} \times Y_{j(l+1)}})|^{\mathcal{M} \times \mathcal{N}} 
	+ \delta_{l+1}.  
\]
We let $\psi_{l+1}$ be the right-hand side.  
Then it is clear that $\psi_{l+1}|_{X_{i(l)} \times \mathcal{N}}$ is $\delta_l$-total.  
The case $l$ is odd is similar, 
and the description of the construction of $\{\psi_l\}$ is completed.  
Now the sequence being decreasing, there exists the limit 
$\psi_\infty \in \operatorname{Apx}_\mathscr{K}^{\vartriangleleft\varphi}(\mathcal{M}, \mathcal{N})$, 
which is clearly of the form $\varphi_\alpha$ 
for some isomorphism $\alpha \colon \mathcal{M} \to \mathcal{N}$ 
by Proposition~\ref{prop:_characterization_of_epsilon-totality}, as desired.  

(iii) $\Rightarrow$ (i) can be verified by the similar argument as above.  
Also, (i) $\Rightarrow$ (ii) $\Rightarrow$ (iii) is trivial.  

(iii) $\Rightarrow$ (iv).  
It follows from (iii) $\Rightarrow$ (ii) that $\mathcal{M}$ is $\mathscr{K}$-universal.  
Let $\iota, \eta \colon \langle \bar{a} \rangle \to \mathcal{M}$ be 
$\mathscr{K}$-admissible embeddings and $\varepsilon$ be a positive real number.  
Then 
\[
	\varphi := \bigl((\varphi_\eta^{}\varphi_\iota^*)
	|_{\iota(\bar{a}) \times \eta(\bar{a})}\bigr)
	|^{\mathcal{M} \times \mathcal{M}} + \varepsilon
\] 
is in $\operatorname{Stx}_\mathscr{K}(\mathcal{M}, \mathcal{M})$, 
so by what we proved in the first paragraph, 
one can find a $\mathscr{K}$-admissible automorphism $\alpha$ of $\mathcal{M}$ with 
$\varphi_\alpha \vartriangleleft \varphi$.  
Since $\varphi\bigl(\iota(a_i), \eta(a_i)\bigr) = \varepsilon$, 
the inequality $d\bigl(\alpha \circ \iota(a_i), \kern 3pt \eta(a_i)\bigr) \leq \varepsilon$ follows.  

(iv) $\Rightarrow$ (ii).  
Suppose that $\mathcal{A}$ is an object of $\mathscr{K}$ and 
$\varphi$ is a strict approximate $\mathscr{K}$-isomorphism from $\mathcal{A}$ to $\mathcal{M}$.  
By assumption, there exists a $\mathscr{K}$-admissible embedding 
$\iota \colon \mathcal{A} \to \mathcal{M} = \overline{\bigcup_n \mathcal{A}_n}$, 
so it suffices to show that there is a $\mathscr{K}$-admissible automorphism $\alpha$ of 
$\mathcal{M}$ with $\varphi_{\alpha \circ \iota} \vartriangleleft \varphi$, 
or equivalently, $\varphi_\alpha^{} \vartriangleleft \varphi\varphi_\iota^*$.  
To see this, find sufficiently large $n$ and finite partial $\mathscr{K}$-isomorphisms 
$\iota_1, \iota_2$ from $\mathcal{A}_n$ into some object $\mathcal{C} $ of $\mathscr{K}$ 
with 
\[
	(\varphi_{\iota_1, \iota_2})|^{\mathcal{M} \times \mathcal{M}} + \varepsilon 
	\vartriangleleft \varphi\varphi_\iota^*.  
\]
Since there exists a $\mathscr{K}$-admissible embedding of $\mathcal{C}$ into $\mathcal{M}$, 
and since $\mathcal{M}$ is approximately $\mathscr{K}$-ultra-homogeneous, 
there exists a $\mathscr{K}$-admissible embedding $\eta \colon \mathcal{C} \to \mathcal{M}$ 
with $d\bigl(b, \eta \circ \iota_2(b)\bigr) < \varepsilon/2$ for 
$b \in \operatorname{dom} \iota_2$.  
Again by the $\mathscr{K}$-ultra-homogeneity of $\mathcal{M}$, 
one can find a $\mathscr{K}$-admissible automorphism $\alpha$ of $\mathcal{M}$ such that 
the inequality $d\bigl(\alpha(a), \kern 3pt \eta \circ \iota_1(a)\bigr) < \varepsilon/2$ 
holds for all $a \in \operatorname{dom} \iota_1$.  Then 
\[
\begin{aligned}
	d\bigl(\alpha(a), b\bigr) 
	&\leq d\bigl(\alpha(a), \kern 3pt \eta \circ \iota_1(a)\bigr) 
	+ d\bigl(\eta \circ \iota_1(a), \kern 3pt \eta \circ \iota_2(b)\bigr) 
	+ d\bigl(\eta \circ \iota_2(b), \kern 3pt b\bigr) \\
	&\leq d\bigl(\iota_1(a), \iota_2(b)\bigr) + \varepsilon 
	= \varphi_{\iota_1, \iota_2}(a, b) + \varepsilon, 
\end{aligned}
\]
which completes the proof.  
\end{proof}

Next, we shall prove the existence of the Fraïssé limit.  
For this, we need the following lemma which claims that, 
in order to see (iii) in Theorem~\ref{thm:_uniqueness_of_fraisse_limit}, 
we only have to check a countable dense part.  

\begin{lem}\label{lem:_checking_dense_part}
Let $\mathcal{M}$ be a $\mathscr{K}$-structure and $M_0$ be 
a countable dense subset of $|\mathcal{M}|$.  
Suppose that, for each $n \in \mathbb{N}$, 
a countable dense subset $\mathscr{K}_{n,0}$ of $\mathscr{K}_n$ is given.  
Then, in order for $\mathcal{M}$ to be the Fraïssé limit of $\mathscr{K}$, 
it is sufficient that, for any $n \in \mathbb{N}$, 
any $\langle \bar{a} \rangle \in \mathscr{K}_{n, 0}$, 
any finite subset $F \subseteq M_0$ and 
any $\varphi \in \operatorname{Stx}_\mathscr{K}\bigl(\langle \bar{a} \rangle, \mathcal{M}\bigr)$ 
which is rational-valued on $\bar{a} \times F$, there exists 
$\psi \in \operatorname{Apx}_\mathscr{K}^{\leq\varphi}
\bigl(\langle \bar{a} \rangle, \mathcal{M}\bigr)$ such that 
$\psi|_{\bar{a} \times \mathcal{M}}$ is $\varepsilon$-total.  
\end{lem}

\begin{proof}
Let $\mathcal{B}$ be an object of $\mathscr{K}$ and 
$\varphi$ be a strict approximate $\mathscr{K}$-isomorphism 
from $\mathcal{B}$ to $\mathcal{M}$, and take 
$\varphi' \in \operatorname{Stx}_\mathscr{K}^{\vartriangleleft\varphi}(\mathcal{B}, \mathcal{M})$.  
Then there exists an arbitrarily large finite subsets $F \in |\mathcal{B}|^n$ 
and $G \subseteq M_0$ and arbitrarily small $\varepsilon > 0$ with 
\[
	\varphi'' := (\varphi'|_{F \times G})|^{\mathcal{B} \times \mathcal{M}} + \varepsilon 
	\vartriangleleft \varphi.  
\]
Without loss of generality, we may assume that $F$ is a generator of $\mathcal{B}$.  
Let $\bar{b} = (b_1, \dots, b_n)$ be an enumeration of $F$.  
Take $\langle \bar{a} \rangle \in \mathscr{K}_{n, 0}$ with 
$d^\mathscr{K}\bigl(\langle \bar{a} \rangle, \langle \bar{b} \rangle\bigr) < \varepsilon/4$ and 
find a joint $\mathscr{K}$-embedding $(\iota, \eta)$ of 
$\langle \bar{a} \rangle$ and $\langle \bar{b} \rangle$ satisfying
$\max_i \varphi_{\iota, \eta}(a_i, b_i) < \varepsilon/4$.  
Then, being a restriction of an extension of a strict approximate $\mathscr{K}$-isomorphism, 
$\varphi''\varphi_{\iota, \eta}$ is strict, so there exists 
$\psi' \in \operatorname{Stx}_\mathscr{K}\bigl(\langle \bar{a} \rangle, \mathcal{M}\bigr)$ 
which is rational-valued on $\bar{a} \times F$ and satisfies 
$\psi' \vartriangleleft \varphi''\varphi_{\iota, \eta}$, 
by Proposition~\ref{prop:_characterization_of_strict_domination}.  
By assumption, we can take $\psi'' \in \operatorname{Apx}_\mathscr{K}^{\leq\psi'}
\bigl(\langle \bar{a} \rangle, \mathcal{M}\bigr)$ such that 
$\psi''|_{\bar{a} \times \mathcal{M}}$ is $\varepsilon/4$-total.  
Then, $(\psi''\varphi_{\eta, \iota})|_{\bar{b} \times \mathcal{M}}$ is 
$\varepsilon/2$-total, and 
\[
	(\psi''\varphi_{\eta, \iota})|_{\bar{b} \times \mathcal{M}}
	\leq (\varphi''\varphi_{\iota, \eta}\varphi_{\eta, \iota})|_{\bar{b} \times \mathcal{M}} 
	\leq \varphi''|_{\bar{b} \times \mathcal{M}} + \varepsilon/2,   
\]
since $\varphi_{\eta, \iota}|_{\bar{b} \times \langle \bar{a} \rangle}$ is 
$\varepsilon/4$-total.  
Now take a finite subset $H \subseteq |\mathcal{M}|$ 
such that $G$ is included in $H$ and 
$(\psi''\varphi_{\eta, \iota})|_{\bar{b} \times H}$ is $3\varepsilon/4$-total, 
and put 
\[
	\psi := \bigl((\psi''\varphi_{\eta, \iota})|_{\bar{b} \times H}\bigr)
	|^{\mathcal{B} \times \mathcal{M}} + \varepsilon/4.  
\]
Then $\psi|_{\bar{b} \times |\mathcal{M}|}$ is $\varepsilon$-total, and 
\[
	\psi 
	\leq (\varphi''|_{\bar{b} \times H})|^{\mathcal{B} \times \mathcal{M}} + 3\varepsilon/4 
	\leq (\varphi'|_{F \times G})|^{\mathcal{B} \times \mathcal{M}} + \varepsilon 
	\vartriangleleft \varphi.  
\]
Since this shows that (iii) in Theorem~\ref{thm:_uniqueness_of_fraisse_limit} holds, 
it follows that $\mathcal{M}$ is the Fraïssé limit of $\mathscr{K}$.  
\end{proof}

\begin{thm}\label{thm:_existence_of_fraisse_limit}
Every Fraïssé category has its limit.  
\end{thm}

\begin{proof}
Take a countable dense subset $\mathscr{K}_{n, 0} \subseteq \mathscr{K}_n$ for each $n$.  
In view of Proposition~\ref{prop:_joint_embedding_and_approximate_isomorphism}, 
we can inductively find a $\mathscr{K}$-structure $\mathcal{A}_k$, 
a $\mathscr{K}$-embedding $\iota_{k-1} \colon \mathcal{A}_{k-1} \to \mathcal{A}_k$ and 
a countable dense subset $A_{k, 0} \subseteq |\mathcal{A}_k|$ so that, 
if $\langle \bar{a} \rangle$ is in $\mathscr{K}_{n, 0}$, 
if $F$ is a finite subset of $A_{k, 0}$, 
and if $\varphi$ is a strict approximate $\mathscr{K}$-isomorphism 
from $\langle \bar{a} \rangle$ to $\mathcal{A}_k$ which is rational-valued on $\bar{a} \times F$, 
then there exists a $\mathscr{K}$-embedding 
$\iota \colon \langle \bar{a} \rangle \to \mathcal{A}_l$ for some $l > k$ with 
$\varphi_\iota^{} \vartriangleleft \varphi_{\iota_{l, k}}\varphi$, 
where $\iota_{l, k}$ denotes the composition of $\iota_k, \dots, \iota_{l-1}$.  
Then the $\mathscr{K}$-structure obtained from the inductive system 
satisfies the assumption in the previous lemma, 
so we are done.  
\end{proof}

\begin{rem}\label{rem:_the_gap}
Let $\mathscr{K}$ be a Fraïssé category.  
If every $L$-embedding between objects of $\mathscr{K}$ is 
a morphism of $\mathscr{K}$, then $\mathscr{K}$ is 
a Fraïssé class in the sense of~\cite{eagle16:_fraisse_limits}.  
We notice that there is a subtle difference between our result
(Theorems \ref{thm:_uniqueness_of_fraisse_limit} 
and~\ref{thm:_existence_of_fraisse_limit}) 
and~\cite[Theorem~2.8]{eagle16:_fraisse_limits}.  
We proved that if $\mathcal{M}$ is the limit of a Fraïssé class $\mathscr{K}$, 
if $\iota, \eta \colon \mathcal{A} \to \mathcal{M}$ are 
\emph{$\mathscr{K}$-admissible} embedding of an object $\mathcal{A}$ of $\mathscr{K}$, 
and if $F$ is a finite subset of $|\mathcal{A}|$, then for any $\varepsilon > 0$ 
there exists a ($\mathscr{K}$-admissible) automorphism $\alpha$ of $\mathcal{M}$ 
with $d\bigl(\alpha \circ \iota(a), \kern 3pt \eta(a)\bigr) < \varepsilon$ 
for all $a \in F$. 
On the other hand, it is claimed 
in~\cite[Definition~2.6 and Theorem~2.8]{eagle16:_fraisse_limits} 
that even if $\iota$ and $\eta$ are not $\mathscr{K}$-admissible, 
one can still find an automorphism with the same property.  

In order to obtain the result claimed in~\cite[Theorem~2.8]{eagle16:_fraisse_limits}, 
one might modify the definition of approximate $\mathscr{K}$-isomorphisms as 
following.  First, for $\mathscr{K}$-structures $\mathcal{M}$ and $\mathcal{N}$, 
define $\operatorname{Apx}_{\mathscr{K},2}'(\mathcal{M}, \mathcal{N})$ as
the set of all approximate isometries of the form $\varphi_{\iota,\eta}$, 
where $\iota$ and $\eta$ are finite partial $L$-isomorphisms from 
$\mathcal{M}$ and $\subseteq \mathcal{N}$ into some $\mathcal{A} \in \mathscr{K}$ 
such that the structures $\langle \operatorname{dom} \iota \rangle$, 
$\langle \operatorname{ran} \iota \rangle$, $\langle \operatorname{dom} \eta \rangle$ 
and $\langle \operatorname{ran} \eta \rangle$are in $\mathscr{K}$.  
Then define $\operatorname{Apx}_\mathscr{K}'(\mathcal{M}, \mathcal{N})$ as 
the closure of the set of all approximate isometries which dominate some element of 
$\operatorname{Apx}_{\mathscr{K},2}'(\mathcal{M}, \mathcal{N})$.  
If one could prove Propositions~\ref{prop:_joint_embedding_and_approximate_isomorphism} 
and~\ref{prop:_compositions_of_approximate_isomorphisms} with this modified definition, 
then one would be able to obtain the desired result by 
simply copying the proofs in this paper.  

However, with this modified definition, 
the proof of Proposition~\ref{prop:_compositions_of_approximate_isomorphisms} no longer works.  
The problem lies in the second paragraph, where the proof is reduced to 
the case that all the relevant structures are objects of $\mathscr{K}$.  
To see the difficulty, 
let $\mathcal{M}_1$, $\mathcal{M}_2$ and $\mathcal{M}_3$ be $\mathscr{K}$-structures, 
and $\varphi \in \operatorname{Apx}_\mathscr{K}'(\mathcal{M}_1, \mathcal{M}_2)$ and 
$\psi \in \operatorname{Apx}_\mathscr{K}'(\mathcal{M}_2, \mathcal{M}_3)$ be strict.  
Then there exist finite partial $L$-embeddings 
$\iota_i \colon \mathcal{M}_i \dashrightarrow \mathcal{A} \ (i = 1,2)$ and 
$\eta_j \colon \mathcal{M}_j \dashrightarrow \mathcal{B} \ (j = 2,3)$, 
where $\mathcal{A}$ and $\mathcal{B}$ are members of $\mathscr{K}$, 
and a positive real number $\varepsilon$ such that 
\[
\begin{aligned}
	\varphi' := \varphi_{\iota_1,\iota_2}|^{\mathcal{M}_1 \times \mathcal{M}_2} + &\varepsilon 
	\vartriangleleft \varphi, &  
	\psi' := \varphi_{\eta_2,\eta_3}|^{\mathcal{M}_2 \times \mathcal{M}_3} + &\varepsilon 
	\vartriangleleft \psi.    
\end{aligned}
\]
It is true that there exists an $L$-embedding $\iota$ of a member $\mathcal{C}$ of $\mathscr{K}$ 
into $\mathcal{M}_2$ such that the image of $\iota$ \emph{almost} includes 
both $\operatorname{dom} \iota_2$ and $\operatorname{dom} \eta_1$, 
and in order to reduce the proof, 
it is expected to show that the restrictions $\varphi_\iota^*\varphi'$ and 
$\psi'\varphi_\iota^{}$ are approximate isomorphisms in the sense of the modified definition; 
but how? 

Note that with the original definition (i.e.,~Definition~\ref{dfn:_approximate_isomorphisms}), 
this difficulty could be avoided.  
This is because we could take $\iota$ above so that 
the image \emph{genuinely} includes the domains of both $\iota_2$ and $\eta_1$, 
whence the restrictions are trivially approximate isomorphisms.  
\end{rem}


\section{UHF algebras}
In this section, we give an application of our theory to C*-algebras.  
This is a generalization of the results in Section~3 of~\cite{masumoto16:_jiang_su}.  

We shall consider the language of unital tracial C*-algebra.  
The language $L_{\mathrm{TC}^*}$ consists of the following symbols: 
\begin{itemize}
\item
	two constant symbols $0$ and $1$; 
\item
	an unary function symbol $\lambda$ for each $\lambda \in \mathbb{C}$, 
	which are to be interpreted as multiplication by $\lambda$; 
\item
	an unary function symbol $*$ for involution; 
\item
	a binary function symbol $+$ and ${}\cdot{}$; 
\item
	an unary predicate symbol $\operatorname{tr}$.  
\end{itemize}
Then every unital C*-algebra with a distinguished trace can be considered as 
a metric $L_{\mathrm{TC}^*}$-structure.  
Note that the distance we adopt is the norm distance, 
and that a map between unital C*-algebras with fixed traces are 
$L$-embeddings if and only if 
it is a trace-preserving injective $*$-homomorphism.  

For non-negative integer $p$ and positive integer $n$, 
we shall denote by $\mathcal{A}_{p, n}$ the C*-algebra $C([0, 1]^p, \mathbb{M}_n)$ of 
all $n \times n$ matrix-valued functions on $[0, 1]^p$.  
We note that there are canonical isomorphisms 
$C([0, 1]^p, \mathbb{M}_n) \simeq C([0, 1]^p) \otimes \mathbb{M}_n$, 
$C([0, 1]^p) \otimes C([0, 1]^q) \simeq C([0, 1]^{p+q})$ and 
$\mathbb{M}_n \otimes \mathbb{M}_m \simeq \mathbb{M}_{mn}$, 
so that $\mathcal{A}_{p, n} \otimes \mathcal{A}_{q, m}$ is 
canonically isomorphic to $\mathcal{A}_{p+q, mn}$.  
Now, for a probability Radon measure $\mu$ on $[0, 1]^p$, 
we can define a trace $\tau_\mu$ on $\mathcal{A}_{p, n}$ by 
\[
	\tau_\mu(f) := \int \operatorname{tr}\bigl(f(t)\bigr)\, d\mu(t), 
\]
where $\operatorname{tr}$ is the normalized trace on $\mathbb{M}_n$.  
It can be easily verified that 
every trace on $\mathcal{A}_{p, n}$ is of this form, 
so that we can identify the traces on $\mathcal{A}_{p,n}$ with 
the probability Radon measures on $[0, 1]^p$.  
Since the group $\operatorname{Homeo}([0, 1]^p)$ of homeomorphisms of 
$[0, 1]^p$ canonically acts on the set of probability Radon measures on $[0, 1]^p$, 
it also acts on the traces of $\mathcal{A}_{p, n}$.  
In the sequel, we only consider the traces which are 
in the $\operatorname{Homeo}([0, 1]^p)$-orbit of $\tau_\lambda$, 
where $\lambda$ is the Lebesgue measure.  
Note that such traces are faithful.  

By definition, a \emph{supernatural number} is a formal product 
\[
	\nu = \prod_{p\text{: prime}} p^{n_p}, 
\]
where $n_p$ is either a non-negative integer or $\infty$ for each $p$ such that 
$\sum_p n_p = \infty$.  
Given a supernatural number $\nu$, we shall define a category $\mathscr{K}_\nu$ as 
following.  Let $\mathbb{N}_\nu$ be the set of all natural numbers which formally divides $\nu$.  
\begin{itemize}
\item
	$\operatorname{Obj}(\mathscr{K}_\nu)$ is the class of 
	all the pairs of the form $\langle \mathcal{A}_{p, n}, \tau \rangle$, 
	where $n$ is in $\mathbb{N}_\nu$ and $\tau$ is in the 
	$\operatorname{Homeo}([0, 1]^p)$-orbit of $\tau_\lambda$.  
\item
	$\operatorname{Mor}_\mathscr{K}\bigl(\langle \mathcal{A}_{p, n}, \tau \rangle, 
	\langle \mathcal{A}_{p', n'}, \tau' \rangle\bigr)$ is the set of 
	all (unital trace-preserving injective) diagonalizable $*$-homomorphisms 
	from $\langle \mathcal{A}_{p, n}, \tau \rangle$ to 
	$\langle \mathcal{A}_{p', n'}, \tau' \rangle$.  
\end{itemize}
Here, a $*$-homomorphism $\iota$ from $\mathcal{A}_{p, n}$ to $\mathcal{A}_{p',n'}$ is 
said to be \emph{diagonalizable} if there exist a unitary $v \in \mathcal{A}_{p',n'}$ and 
continuous functions $t_1, \dots, t_k \colon [0, 1]^{p'} \to [0, 1]^p$ such that 
\[
	\iota\bigl(f\bigr)(s) 
	= \operatorname{Ad}(v_s) \Bigl( \operatorname{diag}
	\bigl[f(t_1(s)), \dots, f(t_k(s))\bigr] \Bigr)
\]
for $f \in \mathcal{Z}_{p,q}$ and $s \in [0,1]^{p'}$, 
where $\operatorname{Ad}(v)$ denotes 
the inner automorphism of $\mathcal{A}_{p',n'}$ associated to $v$, 
and $\operatorname{diag}[a_1,\dots,a_n]$ is the block diagonal matrix with $a_i$ as 
its $i$-th block.  
Note that compositions of diagonalizable $*$-homomorphisms are diagonalizable.  

\begin{rem}\label{rem:_counterexample}
Here, we shall give an example of $L_{\mathrm{TC}^*}$-morphism 
between objects of $\mathscr{K}_\nu$
which cannot be approximated by diagonalizable ones 
with respect to point-norm topology.  
We assume that $2$ is in $\mathbb{N}_\nu$ and use 
$\mathbb{D} := \{z \in \mathbb{C} \mid |z| \leq 1\}$ instead of $[0, 1]^2$.  
Define $t_1, t_2 \colon \mathbb{D} \to \mathbb{D}$ by 
\[
\begin{aligned}
	t_1(re^{i\theta}) &:= re^{i\theta/2}, & 
	t_2(re^{i\theta}) &:= -re^{i\theta/2} & 
	\bigl(r \in [0, 1], \ \theta \in [0, 2\pi)\bigr).  
\end{aligned}
\]
Also, let $u \colon \mathbb{D} \to \mathbb{M}_2$ be a unitary-valued function defined by 
\[
	u(re^{i\theta}) := 
	\left\{
	\begin{array}{ll}
		\begin{pmatrix}
			e^{i\theta/4}\cos \frac{\theta}{4} & \sin \frac{\theta}{4} \\
			-e^{i\theta/4}\sin \frac{\theta}{4} & \cos \frac{\theta}{4}
		\end{pmatrix}
		& \bigl(r \neq 0, \ \theta \in [0, \pi)\bigr) \\
		1_{\mathbb{M}_2} & (r = 0), 
	\end{array}
	\right.
\]
and note that $u(r) = 1_{\mathbb{M}_2}$ while 
$\lim_{\theta \to 2\pi-0} u(re^{i\theta}) 
= \left( \begin{smallmatrix} 0 & 1 \\ 1 & 0 \end{smallmatrix} \right)$ if $r \neq 0$.  
Now, given $C(\mathbb{D})$, consider a matrix-valued function 
\[
	\varphi(f) \colon z \mapsto 
	\operatorname{Ad}(u(z)) \bigl(\operatorname{diag}[f(t_1(z)), f(t_2(z))]\bigr).  
\]
Clearly this is continuous on the complement of the non-negative part of the real axis.  
It is also continuous on the positive part of the real axis, 
as the switch of the eigenvalues is offset by the unitary.  
Finally, it is continuous at the origin, because it converges to the scalar matrix 
$f(0)1_{\mathbb{M}_2}$.  Therefore, this matrix-valued function belongs to 
$C(\mathbb{D}, \mathbb{M}_2)$, so that $\varphi$ defines 
a $*$-homomorphism from $C(\mathbb{D}) $ into $C(\mathbb{D}, \mathbb{M}_2)$.  
We can also easily verify that this is unital, injective, 
and trace-preserving with respect to $\tau_\mu$, where 
$\mu$ is the normalized Lebesgue measure on $\mathbb{D}$.  

We shall show that the map $\varphi$ not approximately diagonalizable.  
Indeed, if $\varphi$ is approximately diagonalizable, 
then there is a $*$-homomorphism 
$\psi \colon C(\mathbb{D}) \to C(\mathbb{D}, \mathbb{M}_2)$ of the form 
\[
	\psi(f) = \operatorname{Ad}(v)\bigl(\operatorname{diag}[f \circ t_1', f \circ t_2']\bigr)
\]
for some continuous maps $t_1', t_2' \colon \mathbb{D} \to \mathbb{D}$ and 
a continuous unitary-valued map $\nu$ which satisfies 
$\|\varphi(\operatorname{id}_\mathbb{D}) - \psi(\operatorname{id}_\mathbb{D})\| < 1/2$.  
It follows that the Hausdorff distance between $\{t_1(z), t_2(z)\}$ and 
$\{t_1'(z), t_2'(z)\}$ is less than $1/2$ for all 
$z$ with $|z| = 1$, but this is impossible.  
We note that this is also a counterexample 
of~\cite[Theorem~6.3]{lin09:_approximately_diagonalizing} 
claiming that any unital $*$-homomorphism from 
$C(X)$ to $C(Y, \mathbb{M}_n)$ is approximately diagonalizable 
if $X$ is a compact metric space which is a locally absolute retract and 
$Y$ is a compact metric space with $\operatorname{dim} Y \leq 2$.  
\end{rem}

\begin{lem}\label{lem:_embeddability}
\begin{enumerate}[label=\textup{(\roman*)}]
\item
	For any object $\langle \mathcal{A}_{p,n}, \tau \rangle$ of $\mathscr{K}_\nu$, 
	there exists a $\mathscr{K}_\nu$-isomorphism 
	from $\langle \mathcal{A}_{p,n}, \tau \rangle$ onto 
	$\langle \mathcal{A}_{p,n}, \tau_\lambda \rangle$.  
\item
	For any $p$, there exists a 
	$\mathscr{K}_\nu$-embedding from $\langle \mathcal{A}_{p,n}, \tau \rangle$ 
	into $\langle \mathcal{A}_{1,n}, \tau_\lambda \rangle$.  
\end{enumerate}
\end{lem}

\begin{proof}
(i) Let $\alpha$ be the homeomorphism of $[0, 1]^p$ with
$\alpha_*(\tau_\lambda) = \tau$.  
Then the induced $*$-homomorphism $\alpha^* \colon f \mapsto f \circ \alpha$ is the desired one.  

(ii) We may assume $\tau = \tau_\lambda$ by (i).  
It suffices to show the case $p = 2$.  
Let $\beta \colon [0, 1] \to [0, 1]^2$ be the Hilbert curve~\cite{hilbert91:_stetige_abbildung}, 
which is a surjective continuous map such that any interval of the form 
$[k/4^l, k+1/4^l]$ is sent to a square of the form 
$[k_1/2^l, k_1+1/2^l] \times [k_2/2^l, k_2+1/2^l]$, 
so that $\beta_*(\tau_\lambda) = \tau_\lambda$.  
Then the $*$-homomorphism $\beta^* \colon f \mapsto f \circ \beta$ is the desired one.  
\end{proof}

\begin{lem}\label{lem:_lemma_for_nap}
Suppose that $\iota_1, \iota_2 \colon \langle \mathcal{A}_{p,n}, \tau \rangle \to 
\langle \mathcal{A}_{p',n'}, \tau \rangle$ are 
$\mathscr{K}$-embeddings of the form 
\[
	\iota_i(f) = \operatorname{diag}[f \circ t_{1, i}, \dots, f \circ t_{k, i}].  
\]
If the diameter of the range of $t_{l, i}$ is less than $\delta$ for all $i$ and $l$, 
then there exists a permutation $\sigma \in \mathfrak{S}_k$ such that 
the inequality $\|t_{l, 1} - t_{\sigma(l), 2}\| < 2\delta$ holds for all $l$.  
\end{lem}

\begin{proof}
For each $l$, let $S_l$ be the set of all $l'$ with 
$\operatorname{Im} t_{l, 1} \cap \operatorname{Im} t_{l', 2} \neq \varnothing$.  
Then, for any $F \subseteq \{1, \dots, k\}$, we have 
\[
	\Bigl| \bigcup_{l \in F} S_l \Bigr| 
	= k\sum_{l' \in \bigcup_{l \in F} S_l} \tau\bigl(\operatorname{Im} t_{l', 2}\bigr) 
	\geq k\tau\Bigl(\bigcup_{l \in F} \operatorname{Im} t_{l, 1}\Bigr) 
	\geq |F|, 
\]
since $\iota_1$ and $\iota_2$ are trace-preserving.  
By Hall's marriage theorem there exists a permutation $\sigma \in \mathfrak{S}_k$ 
with $t_{\sigma(l),2} \in S_l$ for all $l$.  
Now the inequality $\|t_{l,1} - t_{\sigma(l), 2}\| < 2\delta$ is clear.  
\end{proof}

\begin{thm}\label{thm:_knu_is_a_fraisse_category}
The category $\mathscr{K}_\nu$ is a Fraïssé category.  
\end{thm}

\begin{proof}
JEP is a direct consequence of Lemma~\ref{lem:_embeddability} and the fact that 
if $n$ divides $n'$, then there exists a $\mathscr{K}_\nu$-embedding 
from $\langle \mathcal{A}_{1,n}, \tau_\lambda \rangle$ to 
$\langle \mathcal{A}_{1,n'}, \tau_\lambda \rangle$ defined by 
$f \mapsto \operatorname{diag}[f, \dots, f]$.  
For NAP, let $\iota_i$ be $\mathscr{K}_\nu$-embeddings from 
$\langle \mathcal{A}_{p_0,n_0}, \tau_0 \rangle$ into 
$\langle \mathcal{A}_{p_i,n_i}, \tau_i \rangle$ 
for $i = 1, 2$, and suppose that a finite subset $F$ of $\mathcal{A}_{p_0, n_0}$ and 
a positive real number $\varepsilon > 0$ are given.  
Our goal is to find $\mathscr{K}_\nu$-embeddings 
$\eta_i$ from $\langle \mathcal{A}_{p_i,n_i}, \tau_i \rangle$ into 
some object $\langle \mathcal{A}_{p_3,n_3}, \tau_3 \rangle$ such that the inequality 
$d\bigl(\eta_1 \circ \iota_1(f), \kern 3pt \eta_2 \circ \iota_2(f)\bigr) < \varepsilon$ 
holds for all $f \in F$.  
To see this, take $\delta > 0$ so that $|t-t'| < \delta$ implies 
$\|f(t) - f(t')\| < \varepsilon$.  
Apply JEP to find $\mathscr{K}_\nu$-embeddings $\eta_i'$ 
from $\langle \mathcal{A}_{p_i,n_i}, \tau_i \rangle$ into 
some object $\langle \mathcal{A}_{p',n'}, \tau' \rangle$.  
By Proposition~\ref{lem:_embeddability}, we may assume without loss of generality 
that $\tau' = \tau_\lambda$ and $p' = 1$.  
Now, Since $\eta_i' \circ \iota_i$ is a $\mathscr{K}_\nu$-isomorphism, it is of the form 
\[
	\eta_i' \circ \iota_i(f) = 
	\operatorname{Ad}(v_i') \bigl(\operatorname{diag}
	[f \circ t_{1,i}', \dots, f \circ t_{k',i}']\bigr).  
\]
Take sufficiently large natural number $m$ such that $n'm$ is in $\mathbb{N}_\nu$ and 
$|s-s'| < 1/m$ implies $|t_{l,i}'(s) - t_{l,i}'(s')| < \delta/2$ for all $l$ and $i$.  
Define $r_c \colon [0, 1] \to [0, 1]$ by 
$r_c(x) := (x+c-1)/m$ for $c = 1, \dots, m$, and 
let $\rho$ be a $\mathscr{K}_\nu$-embedding from 
$\langle \mathcal{A}_{1,n'}, \tau_\lambda \rangle$ into 
$\langle \mathcal{A}_{1,n'm}, \tau_\lambda \rangle$ of the form
\[
	\rho(f) = \operatorname{diag}[f \circ r_1, \dots, f \circ r_m].  
\]
Then $\rho \circ \eta_i' \circ \iota_i$ is of the form 
\[
	\rho \circ \eta_i' \iota_i(f) = 
	\operatorname{Ad}(v_i) \bigl(\operatorname{diag}
	[f \circ t_{1,i}, \dots, f \circ t_{k,i}]\bigr), 
\]
where the diameter of the image of $t_{l,i}$ is less than $\delta/2$ for all $l$ and $i$.  
By Lemma~\ref{lem:_lemma_for_nap}, 
we may assume without loss of generality that the inequality 
$\|t_{l,1} - t_{l,2}\| < \delta$ holds for all $l$.  
It can be easily verified that $\eta_1 := \rho \circ \eta_1'$ and 
$\eta_2 \colon \operatorname{Ad}(v_1v_2^*) \circ \rho \circ \eta_2'$ are 
the desired $\mathscr{K}_\nu$-embeddings.  

WPP is clear, because up to $\mathscr{K}_\nu$-isomorphisms, 
there are only countable many objects in $\mathscr{K}_\nu$.  
Also, CCP automatically follows from the fact that all the relevant functions are 
$1$-Lipschitz on the unit ball.  
\end{proof}

We shall find a concrete description of the limit of $\mathscr{K}_\nu$.  
For this, the following proposition is useful.  

\begin{prop}\label{prop:_characterization_of_the_limit}
Let $\mathscr{K}$ be a Fraïssé class and 
$\mathcal{M} = \overline{\bigcup_n \mathcal{A}_n}$ be a $\mathscr{K}$-structure.  
Denote by $\iota_{k,j}$ the canonical $\mathscr{K}$-embedding 
from $\mathcal{A}_j$ into $\mathcal{A}_k$.  
Suppose that the following two conditions hold: 
\begin{enumerate}[label=\textup{(\alph*)}]
\item
	Any object $\mathcal{C}$ of $\mathscr{K}$ is 
	$\mathscr{K}$-embeddable into $\mathcal{A}_n$ for some $n$.  
\item
	Given a finite subset $F \subseteq |\mathcal{A}_i|$, 
	a positive real number $\varepsilon$ and a $\mathscr{K}$-embedding 
	$\eta \colon \mathcal{A}_i \to \mathcal{A}_j$ for some $j > i$, 
	one can find $k > j$ and a $\mathscr{K}$-automorphism 
	$\alpha \in \mathcal{A}_k$ such that the inequality 
	$d\bigl(\alpha \circ \iota_{k,j} \circ \eta(a), \iota_{k,i}(a)\bigr) < \varepsilon$ 
	holds for all $a \in F$.  
\end{enumerate}
Then $\mathcal{M} = \overline{\bigcup_n \mathcal{A}_n}$ is the Fraïssé limit of $\mathscr{K}$.  
\end{prop}

\begin{proof}
We shall check (iii) in Theorem~\ref{thm:_uniqueness_of_fraisse_limit}.  
Let $\varepsilon$ be a positive real number, 
$\mathcal{B}$ be an object of $\mathscr{K}$ and $\varphi$ be in 
$\operatorname{Stx}_\mathscr{K}(\mathcal{B}, \mathcal{M})$.  
Then one can find finite subsets $F_1 \subseteq |\mathcal{B}|$ and 
$F_2 \subseteq |\mathcal{A}_i|$, 
an object $\mathcal{C}$ of $\mathscr{K}$, 
and $\mathscr{K}$-embeddings 
$\iota \colon \mathcal{B} \to \mathcal{C}$ and $\eta \colon \mathcal{A}_i \to \mathcal{C}$ 
such that the relation 
\[
	\bigl(\varphi_{\iota,\eta}|_{F_1 \times F_2}\bigr)|^{\mathcal{B} \times \mathcal{M}} 
	\vartriangleleft \varphi
\]
holds.  By assumption~(a), there exists a $\mathscr{K}$-embedding 
$\theta$ of $\mathcal{C}$ into some $\mathcal{A}_j$ with $j > i$.  
Then one can find a $\mathscr{K}$-automorphism $\alpha \in \operatorname{Aut}(\mathcal{A}_k)$ 
for some $k >j$ such that the inequality 
\[
	d\bigl(\alpha \circ \iota_{k,j} \circ \theta \circ \eta(a), \kern 3pt \iota_{k,i}(a)\bigr) 
	< \varepsilon
\]
holds for all $a \in F_2$, by assumption~(b).  
Now, for $b \in F_1$ and $a \in F_2$, we have 
\[
\begin{aligned}
	& d\bigl(\alpha \circ \iota_{k,j} \circ \theta  \circ \iota(b), \iota_{k,i}(a)\bigr) \\
	{}<{}& d\bigl(\alpha \circ \iota_{k,j} \circ \theta  \circ \iota(b), 
	\alpha \circ \iota_{k,j} \circ \theta \circ \eta(a)\bigr) + \varepsilon \\
	{}={}& d\bigl(\iota(b), \eta(a)\bigr) + \varepsilon, 
\end{aligned}
\]
whence 
\[
	\varphi_{\alpha \circ \iota_{k,j} \circ \theta \circ \iota} 
	\leq \bigl(\varphi_{\iota,\eta}|_{F_1 \times F_2}\bigr)|^{\mathcal{B} \times \mathcal{M}} 
	\vartriangleleft \varphi, 
\]
which completes the proof.  
\end{proof}

\begin{cor}\label{cor:_limit_description}
Let $\mathcal{M} = \overline{\bigcup_j \langle \mathcal{A}_{p_j,n_j},\tau_j\rangle}$ be 
a $\mathscr{K}_\nu$-structure and $\iota_{k,j}$ denote 
the canonical $\mathscr{K}$-embedding from $\langle \mathcal{A}_{p_j,n_j}, \tau_j \rangle$ 
into $\langle \mathcal{A}_{p_k,n_k}, \tau_k \rangle$.  
Suppose the following conditions hold: 
\begin{enumerate}[label=\textup{(\alph*)}]
\item
	$p_j \geq 1$.  
\item
	For any $n \in \mathbb{N}_\nu$, there exists $j \in \mathbb{N}$ such that 
	$n$ divides $n_j$.  
\item
	For any $j \in \mathbb{N}$ and $\varepsilon > 0$ 
	there exists $k > j$ such that $\iota_{k,j}$ is of the form 
	\[
		\iota_{k,j}(f) = \operatorname{Ad}(v)\bigl(\operatorname{diag}
		[f \circ t_1, \dots, f \circ t_m]\bigr), 
	\]
	where the diameter of the image of $t_l$ is less than $\varepsilon$ for all $l$.  
\end{enumerate}
Then $\mathcal{M} = \overline{\bigcup_j \langle \mathcal{A}_{p_j,n_j},\tau_j\rangle}$ is 
the Fraïssé limit of $\mathscr{K}_\nu$.  
\end{cor}

\begin{proof}
This is immediate from Lemmas \ref{lem:_embeddability} and~\ref{lem:_lemma_for_nap} 
and Proposition~\ref{prop:_characterization_of_the_limit}.  
\end{proof}

Take an increasing sequence $\{n_j\} \subseteq \mathbb{N}_\nu$ so that 
(b) in~\ref{cor:_limit_description} is satisfied.  
Define $\iota_i \colon \mathcal{A}_{1,n_j} \to \mathcal{A}_{1,n_{j+1}}$ as 
the $*$-homomorphism of the form 
\[
	\iota_i(f) = \operatorname{diag}[f \circ r_1, \dots, f \circ r_m], 
\]
where $r_1, \dots, r_m$ are as in the proof of Theorem~\ref{thm:_knu_is_a_fraisse_category}.  
Then the diagram 
\[
\xymatrix{
	\langle \mathcal{A}_{1,n_1}, \tau_\lambda \rangle \ar[r]^{\iota_1} & 
	\langle \mathcal{A}_{1,n_2}, \tau_\lambda \rangle \ar[r]^{\iota_2} & 
	\langle \mathcal{A}_{1,n_3}, \tau_\lambda \rangle \ar[r]^{\iota_3} & \cdots \\
	\langle \mathbb{M}_{n_1}, \operatorname{tr} \rangle \ar[r] \ar[u] & 
	\langle \mathbb{M}_{n_2}, \operatorname{tr} \rangle \ar[r] \ar[u] & 
	\langle \mathbb{M}_{n_3}, \operatorname{tr} \rangle \ar[r] \ar[u] & \cdots
}
\]
commutes, where $\mathbb{M}_{n_j}$ is canonically identified with 
the C*-subalgebra of constant functions on the interval $[0,1]$.  
Since the upper inductive system satisfies the assumption of~\ref{cor:_limit_description} 
and the limit of the lower inductive system is clearly dense in that of the upper one, 
it follows that the Fraïssé limit of $\mathscr{K}_\nu$ is 
isomorphic to the UHF algebra of type $\nu$ as C*-algebras 
(See \cite[Example~III.5.1]{davidson96:_cstar_algebras} for the definition).  

We conclude this section by showing that 
all $L_{\mathrm{TC}^*}$-embeddings into 
the Fraïssé limit of $\mathscr{K}_\nu$ is indeed 
$\mathscr{K}_\nu$-admissible, so that
the gap explained in Remark~\ref{rem:_the_gap} disappears in this case.  
To see this, we use the following 
lemmas~\cite[Exercise~II.8 and Lemma~III.3.2]{davidson96:_cstar_algebras}.  

\begin{lem}\label{lem:_perturbation_of_functional_calculus}
Let $f$ be a continuous function on a compact subset $X$ of $\mathbb{C}$.  
Then, for any $\varepsilon > 0$, there exists $\delta > 0$ such that 
if $a$ and $b$ are normal elements of a C*-algebra $\mathcal{A}$ 
with $\|a-b\| < \delta$, then $\|f(a)-f(b)\| < \varepsilon$.  
\end{lem}

\begin{lem}\label{lem:_perturbation_of_finite_dimensional_algebras}
For any $\varepsilon > 0$ and $n \in \mathbb{N}$, 
there exists $\delta > 0$ such that if $\mathcal{A}$ and $\mathcal{B}$ are 
C*-subalgebras of a unital C*-algebra $\mathcal{D}$, 
if $\operatorname{dim} \mathcal{A}$ is less than $n$, 
and if $\{e_{ij}^{(k)}\}$ is a system of matrix units which spans $\mathcal{A}$ 
and satisfies $d(e_{ij}^{(k)}, \mathcal{B}) < \delta$, 
then there exists a unitary $u$ in $\mathcal{D}$ with $\|u-1\| < \varepsilon$ and 
$\operatorname{Ad}(u)[\mathcal{A}] \subseteq \mathcal{B}$.  
\end{lem}

\begin{lem}\label{lem:_almost_commuting_element_is_close_to_commuting_element}
Let $\{e_{ij}\}$ be the system of standard matrix units of $\mathbb{M}_n$ and 
$a$ be an element of $\mathbb{M}_m \otimes \mathbb{M}_n$ satisfying 
$\|a(1 \otimes e_{ij}) - (1 \otimes e_{ij})a\| < \varepsilon$.  
Then the inequality $\|a - \bigl(1 \otimes \operatorname{tr}\bigr)(a)\| < n^2\varepsilon$ holds.  
\end{lem}

\begin{proof}
If $a$ is represented as $\sum a_{ij} \otimes e_{ij}$, 
then one can easily verify the inequality 
\[
	\|a_{ij} \otimes e_{ij} - \delta_{ij} \sum_k a_{kk} \otimes e_{kk}/n\| < \varepsilon, 
\]
from which the conclusion follows.  
\end{proof}

\begin{thm}\label{thm:_every_embedding_is_k-admissible}
Every $L$-embedding from an object of $\mathscr{K}_\nu$ into 
the Fraïssé limit of $\mathscr{K}_\nu$ is $\mathscr{K}_\nu$-admissible.  
\end{thm}

\begin{proof}
Let $\mathcal{M}$ be the Fraïssé limit of $\mathscr{K}_\nu$ and 
$\iota \colon \langle \mathcal{A}_{p,n}, \tau \rangle \to \mathcal{M}$ be an $L$-embedding.  
Our goal is to show that $\iota$ can be approximated by $\mathscr{K}$-embeddings 
with respect to the topology of pointwise convergence.  
For simplicity, we only show the case $p = 1$ and $\tau = \tau_\lambda$.  
Set 
\[
	G := \{ 1 \otimes e_{ij} \mid i,j = 1, \dots n \} \cup \{ \operatorname{id}_{[0,1]} \otimes 1 \} 
	\subseteq C[0,1] \otimes \mathbb{M}_n \simeq \mathcal{A}_{1,n}, 
\]
where $\{e_{ij}\}$ is the system of standard matrix units of $\mathbb{M}_n$, 
and note that $G$ is a generator of $\mathcal{A}_{1,n}$.  
Given $\varepsilon > 0$, it suffices to find a $\mathscr{K}$-embedding $\eta$ of
$\langle \mathcal{A}_{1,n}, \tau_\lambda \rangle$ into $\mathcal{M}$ satisfying 
$\|\iota(g) - \eta(g)\| < \varepsilon $ for all $g \in G$.  
For this, take $N \in \mathbb{N}$ with $1/N < \varepsilon/6$ and $nN \in \mathbb{N}_\nu$.  
For $c,d \in \mathbb{N}$ with $0 \leq c < d \leq N$, define 
a continuous function $f_{c,d}$ on $[0,1]$ by 
\[
	f_{c, d}(t) := \left\{
	\begin{array}{ll}
		0 & (t \notin [(c-1)/N, (d+1)/N]) \\
		1 & (t \in [c/N, d/N]) \\
		Nt-c+1 & (t \in [(c-1)/N, c/N]) \\
		-Nt+d-1 & (t \in [d/N, (d+1)/N]).  
	\end{array}
	\right.
\] 
Then by Lemma~\ref{lem:_perturbation_of_functional_calculus}, 
there exists positive $\delta < \varepsilon/2$ such that 
if $a$ is a normal element of $\mathcal{M}$ with 
$\|a - \iota(\operatorname{id}_{[0,1]} \otimes 1 )\| < \delta$, 
then the inequality 
$\|f_{c,d}(a) - \iota(f_{c,d} \otimes 1)\| < 1/N$ holds 
for all $c, d \in \mathbb{N}$ with $0 \leq c < d \leq N$.  
Take such $\delta$ and set $\delta' := \delta/(6n^2+1)$.  

Let 
\[
\xymatrix{
	\langle \mathcal{A}_{1,n_1}, \tau_\lambda \rangle \ar[r]^{\iota_1} & 
	\langle \mathcal{A}_{1,n_2}, \tau_\lambda \rangle \ar[r]^{\iota_2} & 
	\langle \mathcal{A}_{1,n_3}, \tau_\lambda \rangle \ar[r]^{\iota_3} & \cdots \\
	\langle \mathbb{M}_{n_1}, \operatorname{tr} \rangle \ar[r] \ar[u] & 
	\langle \mathbb{M}_{n_2}, \operatorname{tr} \rangle \ar[r] \ar[u] & 
	\langle \mathbb{M}_{n_3}, \operatorname{tr} \rangle \ar[r] \ar[u] & \cdots
}
\]
be the inductive system we saw before Lemma~\ref{lem:_perturbation_of_functional_calculus}.  
Then, by Lemma~\ref{lem:_perturbation_of_finite_dimensional_algebras}, 
there exists a unitary $u$ in $\mathcal{M}$ with $\|u - 1\| < \delta'$ and  
$e_{ij}' := u\bigl[\iota(1 \otimes e_{ij})\bigr]u^* \in \bigcup_k \mathbb{M}_{n_k}$.  
We shall denote by $\mathcal{B}$ the finite dimensional simple C*-subalgebra  
generated by $\{e_{ij}'\}$.  
Note that the inequality 
$\|\iota(1 \otimes e_{ij}) - e_{ij}'\| < 2\delta' \leq \varepsilon$ holds for all $i, j$.  
Also, if $\mathcal{B}$ is included in $\mathbb{M}_{n_k}$, 
then $\mathbb{M}_{n_k}$ is canonically isomorphic to $\mathcal{B} \otimes \mathbb{M}_{n_k/n}$.  
Now, take $a \in \bigcup_k \mathbb{M}_{n_k}$ with 
$\|a - \iota(\operatorname{id}_{[0,1]} \otimes 1)\| < \delta'$.  
By Lemma~\ref{lem:_perturbation_of_functional_calculus}, 
we may assume without loss of generality that $a$ is a positive element with $\|a\| \leq 1$.  
Then $\|ae_{ij}' - e_{ij}'a\| < 6\delta$, 
so by Lemma~\ref{lem:_almost_commuting_element_is_close_to_commuting_element}, 
there exists a positive element $a' \in \bigcup_k \mathbb{M}_{n_k}$ 
which commutes with every element of $\mathcal{B}$ and satisfies the inequalities 
$\|a' - \iota(\operatorname{id}_{[0,1]} \otimes 1)\| < (6n^2+1)\delta' \leq \delta$ and 
$\|a'\| \leq 1$.  
By definition of $\delta$, we have $\|f_{c,d}(a') - \iota(f_{c,d} \otimes 1)\| < 1/N$ 
for $0 \leq c < d \leq N$.   

Let $k_0$ be sufficiently large so that 
both $\mathcal{B}$ and $a'$ is included in $\mathbb{M}_{n_{k_0}}$ 
and $m := n_{k_0}/n$ is a multiple of $N$.  
Since the commutant $\mathcal{B}' \cap \mathbb{M}_{n_{k_0}}$ is canonically isomorphic to 
$\mathbb{M}_m$, the positive element $a'$ can be identified with 
a diagonal matrix of $\mathbb{M}_m$, say $\operatorname{diag}[t_1, \dots, t_m]$.  
Without loss of generality, we may assume $t_1 \leq \dots \leq t_m$.  
Then we have 
\[
\begin{aligned}
	\operatorname{tr}\bigl(\operatorname{diag}\bigl[f_{c,d}(t_1), \dots, f_{c,d}(t_m)\bigr]\bigr) 
	&= \operatorname{tr}^\mathcal{M}\bigl(f_{c,d}(a')\bigr) \\
	&\geq \tau_\lambda(f_{c,d} \otimes 1) - 1/N = (d-c)/N.  
\end{aligned}
\]
This inequality together with Hall's marriage theorem implies that the real numbers 
$t_{mc/N+1}, \dots, t_{m(c+1)/N}$ are included in $[(c-1)/N, (c+2)/N]$.  
Consequently, the element 
\[
	a'' := \operatorname{diag}[r_1, \dots, r_m] \otimes 1 \in 
	C([0,1], \mathbb{M}_m) \otimes \mathcal{B} \simeq \mathcal{A}_{1,n_{k_0}}
\]
satisfies $\|a'' - a'\| \leq 3/N < \varepsilon/2$, 
where $r_1, \dots, r_m$ are as in the proof of Theorem~\ref{thm:_knu_is_a_fraisse_category}.  
One can easily check that the $\mathscr{K}$-embedding 
$\eta \colon \mathcal{A}_{1,n} \to \mathcal{A}_{1,n_{k_0}}$ defined by
\[
\begin{aligned}
	\eta(1 \otimes e_{ij}) &:= e_{ij}', & \eta(\operatorname{id}_{[0,1]} \otimes 1) &= a''
\end{aligned}
\]
has the desired property.  
\end{proof}

\noindent
\textbf{Acknowledgement.} The author would like to thank Alessandro Vignati and Bradd Hart 
for helpful conversations.  This work was supported by Research Fellow of the JSPS 
(no.~26--2990) and the Program for Leading Graduate Schools, MEXT, Japan.

\end{document}